\documentclass[a4paper,12pt]{article}
\usepackage{amsmath,amsthm,amssymb,latexsym}

\usepackage[usenames]{color}
\usepackage[ansinew]{inputenc}

\definecolor{rojo}{rgb}{1,0,0}
\definecolor{blanco}{rgb}{1,1,1}

\let \o=\omega

\title{{\bf New estimates for the maximal singular integral \\
}}
\author{\Large{\Large Joan Mateu, Joan Orobitg, Carlos P\'{e}rez and Joan Verdera}}
\newtheorem{teo}{Theorem}
\newtheorem{teor}{Theorem}
\newtheorem{cor}[teo]{Corollary}
\newtheorem{lemma}[teo]{Lemma}

\newtheorem*{lemanom}{Lemma}

\theoremstyle{definition}
\newtheorem*{gracies}{Acknowledgements}

\newcommand{\Rn}{{\mathbb R}^n}

\newcommand{\C}{\mathbb{C}}
\newcommand{\R}{\mathbb R}
\newcommand{\BC}{\Rn \setminus \overline B}
\newcommand{\n}{\frac{n}{2}}

\begin{document}
\date{}

\maketitle

\begin{abstract}

In this paper we pursue the study of the problem of controlling the
maxi\-mal singular integral $T^{*}f$ by the singular integral $Tf$.
Here $T$ is a smooth homogeneous Calder\'{o}n-Zygmund singular integral
of convolution type. We consider two forms of control, namely, in
the $L^2(\Rn)$ norm and via pointwise estimates of $T^{*}f$ by
$M(Tf)$ or $M^2(Tf)$\,, where $M$ is the Hardy-Littlewood maximal
operator and $M^2=M \circ M$ its iteration. It is known that the
parity of the kernel plays an essential role in this question. In a
previous article we considered the case of even kernels and here we
deal with the odd case. Along the way, the question of estimating
composition operators of the type $T^\star \circ T$ arises. It turns
out that, again, there is a remarkable difference between even and
odd kernels. For even kernels we obtain, quite unexpectedly, weak
$(1,1)$ estimates, which are no longer true for odd kernels.  For
odd kernels we obtain sharp weaker inequalities involving a weak
$L^1$ estimate for functions in $L\,LogL$.
\end{abstract}

\section{Introduction: the model examples}

In this paper we prove new estimates for the maximal singular
integral associated with a singular integral of Calder\'{o}n-Zygmund
type. We start by considering two model examples, the Hilbert
transform and the Beurling transform. The Hilbert transform is the
linear operator defined for almost every $x \in \R$ by the principal
value integral
$$
Hf(x) = p.v.\,  \,\int_{\R} \frac{f(y)}{y-x}\,dy \,,
$$
where $f$ is a function in some $L^p(\R)\,, 1 \leq p < \infty\,,$
and the maximal Hilbert transform is
$$
H^*f(x) = \sup_{\epsilon>0}  \left| \int_{ |y-x|>\epsilon}
\frac{f(y)}{y-x}\,dy \right|\,, \quad x \in \R\,.
$$

The Beurling transform is the one complex variable analog of the
Hilbert transform, that is,
$$
Bf(z) = p.v.\, \int_{\C} \frac{f(\o)}{(\o-z)^2}\,d\o \,,
$$
where $f$ is in some $L^p(\C)\,, 1 \leq p < \infty\,,$ and the
maximal Beurling transform is
$$
\qquad B^*f(z) =  \sup_{\epsilon>0}  \left| \int_{ |\o-z|>\epsilon}
\frac{f(\o)}{(\o-z)^2}\,d\o \right|\,, \quad z \in \C \,.
$$

 Our motivation comes from classical Cotlar's
pointwise estimate
\begin{equation}\label{CotlarClassicalpointwise}
T^*(f)(x)\leq C\,\left(M(Tf)(x) + M(f)(x)\right)
\end{equation}
where $T$ is any Calder\'{o}n-Zygmund singular operator, not necessarily
of convolution type, and $M$ is the standard Hardy-Littlewood
maximal function. (See the original result in \cite{Cot} and the
modern treatment in \cite[p.~185]{GrMF}).
It has been shown in \cite{MV} that it is possible, in some cases,
to improve this estimate by removing $M(f)$ in the right hand side
of \eqref{CotlarClassicalpointwise}. For example, for the Beurling
transform one gets
\begin{equation}\label{pointwiseB}
B^*(f)\leq C\,M(Bf)\,.
\end{equation}
 It follows from the same paper that the analogous estimate for the
Hilbert transform
\begin{equation}\label{pointwiseH1}
H^*(f)\leq C\,M(Hf)
\end{equation}
does not hold. In this paper we show that the right substitute for
the inequality above is
\begin{equation}\label{pointwiseH2}
H^*(f)\leq C\,M^2(Hf)
\end{equation}
where $M^2 = M \circ M$ is the iterated Hardy-Littlewood maximal
operator.

The crucial property to derive \eqref{pointwiseB} is the even
character of the kernel defining $B$. For further developments in
this direction see \cite{MOV}, where one characterizes those even
smooth homogeneous Calder\'{o}n-Zygmund kernels for which the pointwise
estimate \eqref{pointwiseB} holds with $B$ replaced by the
convolution operator $T$ associated with the kernel. One
characterizing condition is the $L^2$ estimate,
\begin{equation} \label{L2estimateB}
 \| T^{*} f  \|_2 \leq  C \| T f \|_2\,,\quad f \in L^2(\Rn)
 \,,
\end{equation}
an apparently weaker condition. Another description is expressed in
terms of a purely algebraic condition involving the spherical
harmonics expansion of the kernel. In particular, even higher order
Riesz transforms $T$ do satisfy \eqref{pointwiseB}, with $B$
replaced by $T$.

The first purpose of this paper is to pursue this point of view in
the case of odd smooth kernels, for which the model example provided
by the Hilbert transform points towards pointwise inequalities of
the type \eqref{pointwiseH2}. The main result is given in Theorem
\ref{T} (see Section 2 below) where the pointwise inequality
\eqref{pointwiseH2}, with $H$ replaced by $T$\,, is shown to be
equivalent to the $L^2$ estimate \eqref{L2estimateB} and, as in the
even case, to a purely algebraic condition in terms of the spherical
harmonics expansion of the kernel.

The second purpose of this paper is to gain a better understanding
of why \eqref{pointwiseH1} fails and to provide appropriate sharp
substitutes. The failure of \eqref{pointwiseH1} is related to the
endpoint boundedness properties of the composition of the maximal
singular integral operator and the singular integral operator
itself. For instance, for the Hilbert transform we are referring to
the operator of the form
\begin{equation}\label{compoH}
f \rightarrow (H^* \circ H) (f)=H^*(Hf).
\end{equation}
We show in Section \ref{contraejemploHilbert} that, indeed, this
operator is \textbf{not} of weak type $(1,1)$ and, as a consequence,
\eqref{pointwiseH1} cannot hold. On the other hand, we show that
$H^* \circ H$ satisfies an ``$L\log L$" type estimate, namely, that
there is a constant $C$ such that
\begin{equation}\label{LlogLH}
\big|\{x\in\R: H^*(Hf)(x)>t\}\big| \le C\,\int_{\R}
\Phi\left(\frac{|f(x)|}{t}\right)\,dx, \quad t>0
\end{equation}
where $\Phi(t)=t\,\log(e+t)$. This estimate seems to be the right
replacement for \eqref{pointwiseH1}, because of the presence of
$M^2$ in \eqref{pointwiseH2} and because it is well known (see
\cite{P2}) that
\begin{equation}\label{endpointM^2}
\big|\{x\in\Rn: M^2 f(x)>t\}\big| \le C\,\int_{\Rn}
\Phi\left(\frac{|f(x)|}{t}\right)\,dx \,.
\end{equation}
We remark that, since  \,$\| M^2 \|_{L^{1,\infty}}=\infty$\,, the
preceding inequality is sharp.


In fact, we show that the above ``$L\log L$" phenomenon holds for
arbitrary Calder\'{o}n-Zygmund singular integral operators, not
necessarily of convolution type. Specifically,  if \, $T_1$ and
$T_2$ \, are such operators, then $T^*_1 \circ T_2$ and even $T^*_1
\circ T^*_2$ satisfy inequalities similar to \eqref{LlogLH} (see
Theorem \ref{teoremamaximal} in Section 2 below). However, there are
special situations, always associated to even kernels, in which one
gets a weak type $(1,1)$ inequality. For example, for the Beurling
transform $B$ one has
\begin{equation*}\label{weakB*B}
\big|\{z\in\C: B^*(Bf)(z)>t\}\big| \le \frac{C}{t} \,\int_{\C}
|f(z)|\,dA(z), \quad t>0\,,
\end{equation*}
$dA$ being two dimensional Lebesgue measure. The explanation is that
even operators enjoy an extra cancellation property smoothing out
the composition.

\section{Main results}

\subsection{The pointwise estimate for odd
kernels}\label{oddkernels}

Let $T$ be a smooth homogeneous Calder\'{o}n-Zygmund singular
integral operator on $\Rn$ with kernel
\begin{equation}\label{eq1}
K(x)=\frac{\Omega(x)}{|x|^n},\quad x \in \Rn \setminus \{0\},
\end{equation}
where $\Omega$ is a (real valued) homogeneous function of degree $0$
whose restriction to the unit sphere $S^{n-1}$ is of class
$C^\infty(S^{n-1})$ and satisfies the cancellation property
$$
\int_{|x|=1} \Omega(x)\,d\sigma(x)=0,
$$
$\sigma$ being the normalized surface measure on $S^{n-1}$. Recall
that $Tf$ is the principal value convolution operator
\begin{equation}\label{eq2}
Tf(x)= p.v. \int f(x-y)\,K(y)\,dy \equiv \lim_{\epsilon\rightarrow
0} T^\epsilon f(x),
\end{equation}
where $T^\epsilon$ is the truncation at level $\epsilon$ defined by
$$
T^{\epsilon}f(x)= \int_{| y-x| > \epsilon} f(x-y) K(y) \,dy\, .
$$
It is well known that the limit in \eqref{eq2} exists for almost all
$x$ for $f$ in $L^p(\Rn), 1 \leq p < \infty$.

The operator $T$ is said to be odd (or even) if the kernel is odd
(or even), that is, if  $\Omega(-x)= - \Omega(x),\; x \in
\Rn\setminus\{0\}$ (or $\Omega(-x)= \Omega(x),\; x \in
\Rn\setminus\{0\}$).

Let $T^{*}$ be the maximal singular integral
$$
T^{*}f(x)= \sup_{\epsilon > 0} | T^{\epsilon}f(x)|, \quad x \in \Rn.
$$

Consider the problem of controlling $T^{*} f$ by $Tf$. The most
basic form of control one may think of is the $L^2$ estimate
\begin{equation} \label{eq3}
 \| T^{*} f  \|_2 \leq  C \| T f \|_2\,,\quad f \in L^2(\Rn) \,.
\end{equation}
 Another way of saying that $T^{*} f$ is dominated by  $Tf$,
much stronger, is provided by the pointwise inequality
\begin{equation}\label{eq4}
T^{*}f(x) \leq C \, M(Tf)(x),\quad x \in \Rn\,,
\end{equation}
where M denotes the Hardy-Littlewood maximal operator. A third form
of control, weaker than \eqref{eq4}, but which still implies the
$L^2$ inequality \eqref{eq3}, is given by the condition
\begin{equation}\label{eq4bis}
T^{*}f(x) \leq C \, M^2(Tf)(x),\quad x \in \Rn\,,
\end{equation}
where $M^2 = M \circ M$ is the iterated Hardy-Littlewood maximal
operator. It was shown in \cite{MV} that the Hilbert transform does
not satisfy \eqref{eq4}, but does satisfy a pointwise inequality
slightly weaker than \eqref{eq4bis} (see \cite[p.959] {MV}).

In this paper we prove that if $T$ is an odd higher order Riesz
transform, then \eqref{eq4bis} holds.  In \cite{MOV} it was shown
that even higher order Riesz transforms satisfy the stronger
inequality \eqref{eq4}. Recall that $T$ is a higher order Riesz
transform if its kernel is given by a function $\Omega$ of the form
$$
\Omega(x)=\frac{P(x)}{|x|^d}, \quad x \in \Rn \setminus \{0\},
$$
with $P$ a homogeneous harmonic polynomial of degree $d \geq 1$. If
$P(x)= x_j$, then one obtains the $j$-th Riesz transform $R_j$. If
the homogeneous polynomial $P$ is not required to be harmonic, but
has still zero integral on the unit sphere, then we call $T$ a
polynomial operator.

Condition \eqref{eq4bis} clearly implies the $L^p$ inequality
\begin{equation*} \label{eq4tris}
 \| T^{*} f  \|_p \leq  C \| T f \|_p\,,\quad f \in L^p(\Rn)\quad 1 < p \le \infty
 \,.
\end{equation*}
As we said before, the Hilbert Transform does not satisfy
\eqref{eq4}. Therefore the presence of the iterated Hardy-Littlewood
maximal operator in the case of odd kernels is in the nature of the
problem.

Our main result states that for odd operators inequalities
\eqref{eq3} and \eqref{eq4bis} are equivalent to an algebraic
condition involving the expansion of $\Omega$ in spherical
harmonics. This condition may be very easily checked in practice and
so, in particular,  we can produce extremely simple examples of odd
polynomial operators for which \eqref{eq3} and \eqref{eq4bis} fail.
For these operators no alternative way of controlling $T^{*}f$ by
$Tf$ is known. To state our main result we need to introduce a piece
of notation.

Recall that $\Omega$ has an expansion in spherical harmonics, that
is,
\begin{equation}\label{eq6}
\Omega(x) = \sum_{j=1}^\infty P_j(x), \quad x \in S^{n-1},
\end{equation}
where $P_j$ is a homogeneous harmonic polynomial of degree $j$. If
$\Omega$ is odd, then only the $P_j$ of odd degree $j$ may be
non-zero.


An important role in this paper will be played by the algebra $A$
consisting of the bounded operators on $L^2(\Rn)$ of the form
$$
\lambda I + S,
$$
where $\lambda$ is a real number and $S$ a smooth homogeneous
Calder\'{o}n-Zygmund operator.

 Our main result reads as follows.

\begin{teor} \label{T}
Let $T$ be an odd smooth homogeneous Calder\'{o}n-Zygmund operator
with kernel \eqref{eq1} and assume that $\Omega$ has the expansion
\eqref{eq6}. Then the following are equivalent.
\begin{enumerate}
\item[(i)]
$$
T^{*}f(x) \leq C \, M^2(Tf)(x)\,,\quad x \in \Rn\,.
$$
\item [(ii)]
$$
 \| T^{*} f  \|_2 \leq  C \| T f \|_2\,,\quad f \in L^2(\Rn) \,.
$$
\item [(iii)]
The operator $T$ can be factorized as $T=R\circ U$, where $U$ is an
invertible operator in the algebra $A$ and $R$ is an odd higher
order Riesz transform associated to a harmonic homogeneous
polynomial $P$ which divides each $P_j$ in the ring of polynomials
in $n$ variables with real coefficients .
\end{enumerate}
\end{teor}

Two remarks are in order.


{\bf{Remark 1.}} As in \cite{MOV}, $(iii)$ can be reformulated in a
more concrete fashion as follows. Assume that the expansion of
$\Omega$ in spherical harmonics is
$$
\Omega(x) = \sum_{j=j_o}^\infty P_{2j+1}(x)\,,\,\, P_{2j_0+1} \neq
0\,.
$$
Then $(iii)$ is equivalent to the following

{\it{$(iv)$} For each $j$ there exists a homogeneous polynomial
$Q_{2j-2j_0}$ of degree $2j-2j_0$ such that $P_{2j+1}= P_{2j_0 +1
}\,Q_{2j-2j_0}$ and $\sum_{j=j_o}^\infty \gamma_{2j+1}\,\,
Q_{2j-2j_0}(\xi) \neq 0\,, \quad \xi \in S^{n-1}$.}

Here for a positive integer $j$ we have set
\begin{equation}\label{eq7}
\gamma_j = i^{-j}\, \pi^{\frac{n}{2}}\,
\frac{\Gamma(\frac{j}{2})}{\Gamma ( \frac{n+j}{2})}\,.
\end{equation}
The quantities $\gamma_j$ appear in the computation of the Fourier
multiplier of the higher order Riesz transform $R$ with kernel given
by a homogeneous harmonic polynomial $P$ of degree $j$. One has (see
\cite[p.73]{St})
$$
\widehat{Rf}(\xi)= \gamma_j\,\,\frac{P(\xi)}{|\xi|^j} \,
\hat{f}(\xi),\quad f \in L^2(\Rn)\,.
$$
Throughout this paper the Fourier transform of $f$ is $\hat{f}(\xi)=
\int f(x) e^{-i x\cdot \xi} dx \,, \,\, \xi \in \Rn\,.$

The proof that $(iii)$ and $(iv)$ are equivalent is exactly as in
\cite{MOV}.

{\bf{Remark 2.}} Condition $(iii)$ is rather easy to check in
practice. For instance,  consider the polynomial of third degree
$$
P(x)= x_1 + (n+1) \left(x_1^3 - 3 x_1 x_2^2\right)\,, \quad x \in
S^{n-1}\,,
$$
The polynomial operator associated with $P$ does not satisfy $(i)$
nor $(ii)$, because the definition of $P$ above is also the
spherical harmonics expansion of $P$ and, although $x_1$ divides the
two terms, a calculation shows that $\gamma_1 + \gamma_3
\,(n+1)\,(\xi_1^2- 3\,\xi_2^2)$ vanishes on the sphere. On the other
hand, if $-1 < \lambda < 1$ the polynomial operator associated with
the polynomial
$$
P(x)= x_1 + \lambda\, (n+1) \left(x_1^3 - 3 x_1 x_2^2\right)\,,
\quad x \in S^{n-1}\,,
$$
does satisfy $(i)$ and $(ii)$. Thus, as in the even case, we
conclude that the condition on $\Omega$ so that $T$ satisfies $(i)$
or $(ii)$ is rather subtle.

For the proofs we will rely heavily on \cite{MOV} and the reader
will be assumed to have some familiarity with that paper. The
strategy for the proof is essentially the same as in \cite{MOV}, but
two main differences arise, which require some new ideas. In the
even case, in the proof of ``$(iii)$ implies $(i)$" (``the
sufficient condition") for polynomial operators associated with a
homogeneous polynomial of degree $2N$, the differential operator
$\Delta^N$ plays an essential role. In the odd case the natural
substitute for $\Delta^N$ is a pseudo-differential operator, which
is non local. Thus one loses the support of certain functions. This
is a new difficulty which must be overcome. A second difference is
that one cannot hope to have the subtle $L^\infty$ estimates of
\cite{MOV}, which have to be replaced by $BMO$ estimates. This is,
in some sense, favorable, because proofs are simpler at some points,
just because an $L^\infty$ estimate is not possible and must be
replaced by a straight $BMO$ estimate.

We devote Sections $4$, 5 and 6 to the proof of the sufficient
condition ($(iii)$ implies $(i)$). In Section 4 we prove that the
odd higher order Riesz transforms satisfy $(i)$. Section 5 is
devoted to the proof of the sufficient condition for polynomial
operators.  The drawback of the argument used is that we lose
control on the dependence of the constants on the degree of the
polynomial. The main difficulty we have to overcome in Section 6 to
complete the proof of the sufficient condition in the general case,
is to find a second approach to the polynomial case which gives some
estimates with constants independent of the degree of the
polynomial. This allows the use of a compactness argument to finish
the proof. As in the even case, the approach in Section 4 cannot be
dispensed with, because it provides certain properties which are
vital for the final argument and do not follow otherwise.

In Section 7 we prove the necessary condition, that is, $(ii)$
implies $(iii)$. First we deal with the polynomial case. Analysing
the inequality $(ii)$ via Plancherel at the frequency side we obtain
various inclusion relations among zero sets of certain polynomials.
This requires a considerable computational effort, as in the even
case. In a second step we solve the division problem which leads us
to $(iii)$ by a recurrent argument with some algebraic geometry
ingredients . The question of independence on the degree of the
polynomial appears again, this time related to the coefficients of
certain expansions. Section 8 is devoted to the proof of the
combinatorial lemmas used in the previous sections.

\subsection{Composing maximal singular integrals with singular integrals}

In the previous section we discussed several estimates for operators
of the form $H^*\circ H$ or $B^*\circ B$. We now extend these
results by considering general Calder\'on-Zygmund singular integral
operators. Our point of view is strongly motivated by a celebrated
result of R. Coifman and C. Fefferman from the seventies, Theorem
\ref{CoifFefferman} below. We recall that $A_{\infty}$ is the class
of weights $\cup_{p\ge 1}A_p$.

\begin{teor}\label{teoremamaximal}

Let $T_1$ and $T_2$ be two the Calder\'on-Zygmund singular integral
operators.

a) If $0<p<\infty$ and  $w\in A_{\infty}$,  then there is a
constant $C$ depending on the $A_{\infty}$ constant of $w$ such
that
\begin{equation}\label{coiffeffstrongM2}
\int_{\Rn} (T^*_1\circ T_2)(f)(x)^p\, w(x) dx\leq C\,\int_{ \Rn}
(M^2(f)(x))^p \,w(x)dx,
\end{equation}
and
\begin{equation}\label{coiffeffweakM2}
\sup_{ t >0}  \frac{1}{ \Phi(\frac{1}{t}) }\,w( \{ x\in\Rn:
(T^*_1\circ T_2)(f)(x) > t \} ) \le C\, \sup_{ t
>0} \frac{1}{ \Phi(\frac{1}{t}) } w( \{ x\in \Rn:
M^{2}(f)(x)> t \} )
\end{equation}
where $\Phi(t)= t\log(e+t)$ and  $f$  is a function for which the
left hand side is finite.

b) The estimates \eqref{coiffeffstrongM2} and
\eqref{coiffeffweakM2} hold with $T^*_1 \circ T_2$ replaced by
$T^*_1 \circ T^*_2$ in the left hand side.

\end{teor}

\begin{cor}\label{coroLlogLCZO}
Let $T_1$ and $T_2$ as above and let $w\in A_1$. Then there is a
constant $C$ depending on the $A_1$ constant of $w$ such that
\begin{equation}\label{LlogLCZO}
w( \{x\in\R: T^*_1\circ T_2(f)(x)>t\}) \le C\,\int_{\R}
\Phi\left(\frac{|f(x)|}{t}\right)\,w(x)dx, \quad t>0
\end{equation}
where $\Phi(t)= t\log(e+t)$.

The estimate \eqref{LlogLCZO} holds with $T^*_1 \circ T_2$
replaced by $T^*_1 \circ T^*_2$ in the left hand side.

\end{cor}

As we mentioned before, for general Calder\'on-Zygmund singular
integral operators $T_1$ and $T_2$ their composition $T_1\circ T_2$
is not of weak type $(1,1)$. This should be compared with the case
of Fourier multipliers $T_m$ when the multiplier $m$ satisfies the
classical Mihlin condition. Indeed, by classical well known results,
if $T_{m_1}$ and $T_{m_2}$ are two multipliers the composition
operators $T_{m_1}\circ T_{m_2}=T_{m_1m_2}$ is also of weak type
$(1,1)$\,, and hence is an algebra \cite [2.5.5]{GrCF}. Indeed,
suppose that $T_{m_1}$ and $T_{m_2}$ are two multiplier operators
such that each multiplier $m_j$, $j=1,2$, is bounded, belongs to
$C^{[\frac{n}{2}]+1}$ in the complement of the origin and satisfies
the classical Mihlin condition,
\begin{equation*}\label{mihlin}
\,(\partial^{\alpha}m_j)(\xi) \leq \frac{c}{|\xi|^{\alpha}} \quad
\xi\neq 0
\end{equation*}
for any $\alpha$ such that $|\alpha| \leq [\frac{n}{2}]+1$. Now
consider as above the composition operator $T_{m_1}\circ T_{m_2}$\,,
which is another multiplier operator with multiplier $m_1m_2$. Then
since $m_1 m_2$ satisfies again Mihlin's condition by the Leibnitz
rule, then $T_{m_1}\circ T_{m_2}$ is of weak type $(1,1)$. We recall
here that in either the case of the Hilbert transform (Theorem
\ref{counterHilbert}) or the Riesz transform (\cite{MV})\,, where
the multipliers are smooth, $T_{m_1}$ cannot be replaced  by
$T^*_{m_1}$ in the above result.

%

The Beurling transform, being an even smooth Calder\'{o}n-Zygmund
operator, should enjoy an extra cancellation property that allows
for an improvement of Theorem \ref{teoremamaximal}. This can be
readily verified for the operator $ B^*\circ \overline{B}$, where
$\overline{B}$ denotes the operator whose kernel is the complex
conjugate of the kernel of the Beurling transform. Since
$\overline{B} $ is precisely the inverse of $B$, the pointwise
inequality \eqref{pointwiseB} implies immediately that $ B^*\circ
\overline{B}$ is of weak type $(1,1)$.

It turns out that the operator $B^*\circ B$ is also of  weak type
$(1,1)$, in striking contrast with the fact that $H^*\circ H$ is
not.  This is more difficult to prove and follows from the
pointwise inequality
\begin{equation}\label{compostionBeurlingpointwise}
B^*(B(f))(z)\leq C\, \left((B^2)^*(f)(z) + M(f)(z)\right)\,, \quad
z \in \C \,,
\end{equation}
because $B^2 = B \circ B$ is again a smooth Calder\'{o}n-Zygmund
singular integral operator and hence its maximal operator is of weak
type $(1,1)$. This will be shown in Section
\ref{compostionestimates}. Indeed, we prove there a more general
result which reads as follows.

\begin{teor}\label{higherorderRieszTransforms}
Let $R$ be an even higher order Riesz transform and let $T$ be an
even smooth homogeneous Calder\'{o}n-Zygmund operator. Then there
exists a smooth homogeneous Calder\'{o}n-Zygmund operator $S$ such
that
$$
R^*(T(f))(x)\le C(S^*(f)(x)+M(f)(x)) \quad x\in \Rn.
$$
The operator $S$ is defined by the identity $R\circ T=S+cI,$ \,
where $c$ is an appropriate constant.

In particular, the operator $R^{*} \circ T$ is of weak type
$(1,1)$\,.

\end{teor}

\section{Some preliminaries}

\subsection{Sharp maximal operators}\label{sharp}

For $\delta>0$, let $M_\delta$ be the maximal function
$$M_\delta f(x)=M(|f|^\delta)^{1/\delta}(x)=\left (\sup_{Q\ni
x}\frac{1}{ |Q|}\int_Q |f(y)|^\delta \, dy \right )^{1/\delta}.$$
Also,  let $M^\#$ be the usual sharp maximal function  of Fefferman
and Stein \cite{FS},
$$M^\#(f)(x)=\sup_{Q\ni x}\inf_c \frac{1}{|Q|}\int_Q |f(y)-c|\, dy
\approx \sup_{Q\ni x}\frac{1}{|Q|}\int_Q |f(y)-f_Q|\, dy,  $$
where as usual $f_Q=\frac{1}{|Q|}\int_Q f(y)\, dy$ denotes the
average of $f$ over $Q$.

We also consider the following useful variant of the above sharp
maximal operator
$$M_\delta
^\# f(x)=M^\# (|f|^\delta )(x)^{1/\delta}.$$

The main inequality between these operators to be used is a version
of the classical one due to C. Fefferman and E. Stein (see \cite{Jo}
for a proof simpler than the original, or \cite[p.~148]{GrMF}).

\begin{teor}\label{FeffermanStein}
Let $w$ be an $A_{\infty}$ weight and let $\delta>0$.
\\
a) Let $0<p<\infty$\,.  Then there exists a positive constant $C$
depending on the $A_{\infty}$ condition of $w$ and $p$ such that
\begin{equation*}\label{FSpp}
\int_{\Rn} (M_\delta f(x))^p \, w(x)dx \leq C\,\int_{\Rn}
(M^\#_\delta f(x))^p \, w(x)dx,
\end{equation*}
for every function $f$ such that the left hand side is finite.

b) Let $\varphi :(0,\infty) \rightarrow (0,\infty)$ satisfy the
doubling condition . Then, there exists a constant $C$ depending
upon the $A_{\infty}$ condition of $w$ and the doubling condition
of $\varphi$ such that
\begin{equation*} \label{FSweak}
\sup_{ t >0} \varphi(t) \, w\Big( \{ y\in\Rn: M_{\delta}f(y)
>
t
 \} \Big)
\le C\, \sup_{ t >0} \varphi(t) \,w\Big( \{ y\in\Rn:
M^{\#}_{\delta}f(y)
>
t
 \} \Big)
\end{equation*}
for every function such that the left hand side is finite.
\end{teor}

\subsection{Orlicz spaces and normalized measures}\label{orlicz}

%

We need some few facts from the theory of Orlicz spaces that we
will state without proof.  For more information about these spaces
the reader may consult the recent book by Wilson \cite{W} or
\cite[p.~158]{GrMF}. Let $\Phi:[0,\infty)\to[0,\infty)$ be a Young
function. The $\Phi$-average of a function $f$ over a cube $Q$ is
defined to be the $L_\Phi(\mu)$ norm of $f$ with $\mu$ the
normalized measure of the cube $Q$ and it is denoted by
$\|f\|_{\Phi,Q}$. That is,
$$ \|f\|_{\Phi, Q}=
\inf \{\lambda>0\,:\, \frac{1}{|Q|}\int_Q\Phi \left
(\frac{|f(x)|}{\lambda }\right )dx\leq 1\}. $$
In this paper we will consider the Young functions\, $\Phi(t)=
t\,(1+\log^+t) \approx t\,\log(e+t)$ \,and \,$ \Psi(t)=e^t-1$. The
corresponding averages will be denoted by\, $ \|\cdot\|_{\Phi,Q} =
\|\cdot\|_{L(\log L),Q}$ \,and $\|\cdot\|_{\Psi,Q} = \|\cdot\|_{\exp
L,Q}$\, respectively. We will use the following well known
generalized H\"{o}lder's inequality
\begin{equation}\label{Holder-Orlicz2}
\frac1{|Q|}\,\int_{Q}|f(x)\,g(x)|\,dx \le C\,\|f\|_{\exp
L,Q}\,\|g\|_{L(\log L),Q}.
\end{equation}
In particular, we obtain the following inequality, which will be
used later on in this article,
\begin{equation}\label{John-Nirenberg-LLogL}
\frac{1}{|Q|}\int_Q \left |b(y)-b_Q \right |f(y)\,dy\leq
C\|b\|_{BMO}\|f\|_{L(\log L),Q}.
\end{equation}
for any function $b \in BMO$ and any non negative function $f$. This
inequality follows from \eqref{Holder-Orlicz2} combined with the
classical John-Nirenberg inequality \cite{JN} for $BMO$
functions: there is a dimensional constant $c$  such that
$$
\frac{1}{  |Q| } \int_{Q} \exp ( \frac{ |b(y)- b_{Q}| }{
c\|b\|_{BMO} } )\, dy \le 2
$$
which easily implies that
$$
\|b-b_Q\|_{\exp L,Q} \le  c\, \|b\|_{BMO}.
$$
%


In view of this result and its applications it is natural to define
as in \cite{P2} a maximal operator
$$ 
M_{L(\log L)}f(x) = \sup_{Q\ni x}\|f\|_{L(\log L), Q},$$
where the supremum is taken over all the cubes containing $x$.
(Other equivalent definitions can be found in the literature.) We
will also use the pointwise equivalence
\begin{equation}\label{iteration}
M_{L(\log L)}f(x) \approx M^{2}f(x).
\end{equation}
This equivalence was obtained in \cite{P1} (see \cite{CGMP} for a
different argument) and it relationship with commutators of singular
integrals and $BMO$ functions was studied in \cite{P2} and
\cite{P3}. The sharp endpoint modular inequality for $M^2$, already
mentioned in \eqref{endpointM^2},  will play an important
role.

Finally, we will employ several times  the following simple
Kolmogorov inequality.  Let $0<p<q<\infty$, then there is a constant
$C=C_{p,q}$ such that for any measurable function $f$
\begin{equation*}\label{kolmogorov}
\|f\|_{L^p(Q, \frac{dx}{|Q|})}\leq C\, \|f\|_{L^{q,\infty}(Q,
\frac{dx}{|Q|})}.
\end{equation*}

\subsection{Cotlar's pointwise inequality for Calder\'on-Zygmund operators}

By a Calder\'on-Zygmund operator we mean a continuous linear
operator \\  $T:C_0^{\infty}({\mathbb R}^n)\to\mathcal{D}'({\mathbb
R}^n)$ that extends to a bounded operator on $L^2({\mathbb R}^n)$,
and whose distributional kernel $K$ coincides, away from the
diagonal, with a function $K$ satisfying the size estimate
$$
|K(x,y)|\le \frac{c}{|x-y|^n}
$$
and the regularity condition
$$
|K(x,y)-K(z,y)|+|K(y,x)-K(y,z)|\le
c\frac{|x-z|^{\epsilon}}{|x-y|^{n+\epsilon}},
$$
for some $\epsilon>0$ and whenever $2|x-z|<|x-y|$\,. The kernel of
$T$ is $K$ in the sense that
$$Tf(x)=\int_{{\mathbb R}^n}K(x,y)f(y)dy,$$
whenever $f\in C_0^{\infty}({\mathbb R}^n)$ and
$x\not\in\mbox{supp}(f)$. Let $T^{*}$ be the maximal singular
integral
$$
T^{*}f(x)= \sup_{\epsilon > 0} | T^{\epsilon}f(x)|, \quad x \in \Rn.
$$
where $T^\epsilon$ is the truncation at level $\epsilon$ defined by
$$
T^{\epsilon}f(x)= \int_{| y-x| > \epsilon} K(x,y)f(y)dy,
$$
We refer to \cite[p.~175]{GrMF} for a complete account on these
operators. In the same reference, p. 185, it can be found an
improvement of Cotlar's inequality \eqref{CotlarClassicalpointwise}
that will be useful for our estimates. It reads as follows.
\begin{teor}\label{CotlarPuntual}
Let $T$ and $T^*$ as before and let $0<\delta<1$. Then there is a
positive constant $C =C_\delta$ such that
\begin{equation}\label{CotlarRevisedClassicalpointwise}
T^*(f)(x)\leq C\,M_{\delta}(Tf)(x) + C\,Mf(x), \quad x\in \Rn.
\end{equation}
\end{teor}
Observe that by Jensen's inequality,
\eqref{CotlarRevisedClassicalpointwise} is an improvement of
\eqref{CotlarClassicalpointwise}. Also, it should be mentioned that
A. Lerner has improved this estimate in \cite{Le2}.

The ideas leading to Cotlar estimate
\eqref{CotlarRevisedClassicalpointwise} were crucial to derive the
good-$\lambda$ inequality relating $T^*$ and $M$ found by R. Coifman
and C. Fefferman in \cite{CoF}. In particular we will use in Section
\ref{compostionestimates} the following result.

\begin{teor}\label{CoifFefferman} Let $T$ be any Calder\'{o}n-Zygmund
operator. Then

a)  If $0<p<\infty$ and $w\in A_\infty$, then there exists a
positive constant $C$ depending upon the $A_{\infty}$ condition of
$w$ such that
\begin{equation}\label{strongcoifman-fefferman}
\int_{\Rn} |T^*f(x)|^p\, w(x)\,dx \le C\,\int_{\Rn} Mf(x)^p\,
w(x)\,dx.
\end{equation}

b) Let $\varphi :(0,\infty) \rightarrow (0,\infty)$ doubling. Then,
there exists a positive constant $C$ depending upon the $A_{\infty}$
condition of $w$ and the doubling condition of $\varphi$ such that
\begin{equation*} \label{weakcoifman-fefferman}
\sup_{ t >0} \varphi(t) w( \{ y\in\Rn: |T^*f(x)|
>
t
 \} )
\le C\, \sup_{ t >0} \varphi(t) w( \{ y\in\Rn: Mf(x)
>
t
 \} )
\end{equation*}

\end{teor}

Also we will use a local versi\'{o}n of \eqref{strongcoifman-fefferman}
in the proof of \eqref{fslogl} in Lemma \ref{pointwise} below: if
$0<p<\infty$, $w\in A_\infty$, and  $Q$ is an arbitrary cube, then
there exists a constant $C$ depending upon the $A_{\infty}$
condition of $w$ such that
\begin{equation}\label{localstrongcoifman-fefferman}
\int_{2Q} |T^*f(x)|^p\, w(x)\,dx \le C\,\int_{2Q} Mf(x)^p\,
w(x)\,dx,
\end{equation}
for any function $f$ supported in $ Q$. The proof of this estimate
is an adaptation of the proof in \cite{CoF} by considering
everything at local level. However, it should be mentioned that a
different approach to the above theorem, which may be found in
\cite{AP}, yields the local version.  It is based on the combination
of the well known good-$\lambda$ inequality of Fefferman-Stein
Theorem \ref{FeffermanStein}, which is much simpler, and the
pointwise estimate \eqref{ap} from next lemma which will be used in
the paper. This procedure has been applied in \cite{LOPTT} into the
context of multilinear Calder\'{o}n-Zygmund singular integral operators
to derive sharp results.

\begin{lemma}\label{pointwise}
If $T$ is a Calder\'on-Zygmund singular operator,  then
\begin{equation} \label{fslogl}
M^{\#}(Tf) (x) \le C\, M^2(f)(x)\,,
\end{equation}
and, for $0<\delta <1$\,,
\begin{equation} \label{ap}
M^{\#}_{\delta}(Tf) (x) \le C_{\delta}\, Mf(x)\,.
\end{equation}
\end{lemma}


It is well known that inequality \eqref{fslogl} holds with the right
hand side replaced by the larger operator $M_p(f)$ with $p>1$ (see,
for instance,  \cite[p.~153]{GrMF}). However, this is not sharp
enough for many purposes and an excellent alternative is given by
\eqref{ap} which can be found in \cite{AP}. We sketch the proof of
inequality \eqref{fslogl} in Section \ref{compostionestimates}.

\section{Odd higher order Riesz transforms}
In this section we prove that if $T$ is an odd higher order Riesz
transform, then
\begin{equation}\label{eq8}
T^{*}f(x) \leq C \, M^{2}(Tf)(x),\quad x \in \Rn\,,
\end{equation}
By translating and dilating one reduces the proof of \eqref{eq8} to
\begin{equation*}\label{eq9}
|T^1 f(0)| \leq C \, M^{2}(Tf)(0)\,,
\end{equation*}
where
$$T^1 f(0)= - \int_{| y| > 1} f(y) K(y) \,dy $$
 is the truncated integral at level $1$. Recall that the kernel of our singular integral is
$$
K(x)= \frac{\Omega(x)}{|x|^{n}}= \frac{P(x)}{|x|^{n+d}}\,,
$$
where $P$ is an odd homogeneous harmonic polynomial of degree $d\geq
1$. The argument proceeds along the lines of the even case, but, as
we said above,  two important differences arise. The first is that,
for odd $d$, $(-\Delta )^{d/2}$ is not a differential operator and
this complicates the situation. We will work with the
pseudo-differential operator $(-\Delta )^{1/2} \Delta^N$ , where $d
= 2N + 1$.  The definition of $(-\Delta )^{1/2}$ on test functions
$\Psi$  is $(-\Delta )^{1/2}\Psi =\sum_{j=1}^{n}
R_{j}(\partial_{j}\Psi)$, where the $R_{j}$ are the Riesz transforms
normalized so that $\widehat{R_{j}\Psi}(\xi) = -i \xi_{j}/
|\xi|\,\Psi(\xi)\,.$  The kernel of $R_j$ is then $\rho
\,x_{j}|x|^{-n-1}\,,$ where $\rho $ is a constant which depends only
on the dimension $n$ and whose concrete value is irrelevant in this
paper. On the Fourier transform side we then have $\widehat{(-\Delta
)^{1/2}\Psi} (\xi) = |\xi| \widehat{\Psi}(\xi)$.
 The idea is to obtain an identity of the form
\begin{equation}\label{eq10}
K(x)\chi_{\BC}(x)=  T(b)(x)\,,
\end{equation}
where $B$ is the open ball of radius $1$ centered at the origin and
$b$ is a certain function. To this end, consider a fundamental
solution of $(-\Delta )^{1/2} \Delta^N$, that is, a function $E$
such that $(-\Delta )^{1/2} \Delta^N E=\delta$, where $\delta$ is
the Dirac delta at the origin. One can take $E$ as a solution of
$\Delta^N E = c_{n}/|x|^{n-1}$, where the constant $c_{n}$ is chosen
so that $\widehat{c_{n}/|x|^{n-1}}(\xi)=1/|\xi |$. The formula $c_{n}=\displaystyle\frac{\Gamma(\frac{n-1}{2})}{2\pi^{n/2}\Gamma(\frac{1}{2})}$ wil be used in Section 8. Notice that $E$ can
always be taken to be radial (see Section 8 for a precise
expression). Consider the function
\begin{equation}\label{eq11}
\varphi(x)= E(x)\,\chi_{\BC}(x) + (A_0+A_1\,|x|^2 +...+
A_{2N}\,|x|^{4N})\,\chi_B(x)\,,
\end{equation}
where the constants $A_0, A_1,..., A_{2N}$ are chosen so that the
derivatives of $\varphi$ up to order $2N$ extend continuously to the
boundary of $B$.  Then, in computing the distributional derivatives
of $\varphi$, one can apply $2N+1$ times Green-Stokes' Theorem and
the boundary terms will vanish. This is most conveniently done by
applying $N$ times Corollary 1 and one time Lemma 1 in \cite{MOV}.
The conclusion is that, for some constants $\alpha_j$ and $\beta_k$,
\begin{equation}\label{eq11bis}
\begin{split}
(-\Delta)^{1/2}\Delta^{N} \varphi  & = (-\Delta)^{1/2}\left(
\frac{c_{n}}{|x|^{n-1}}
\chi_{B^{c}}(x)  + ( \alpha_{0} +\alpha_{1}|x|^2 + \dots +\alpha_{N}|x|^{2N}  )  \chi_{B}(x) \right)  \\
 = \sum_{j=1}^{n} R_{j} & \left( c_{n}\,(1-n)\,\frac{x_{j}}{|x|^{n+1}}
\chi_{B^{c}}(x)  +  ( \beta_{1}x_{j}+ \beta_{2}x_{j} |x|^2 +\dots + \beta_{N}x_{j}|x|^{2N-2})   \chi_{B}(x) \right)  \\
& : = b(x),
\end{split}
\end{equation}
where the last identity is a definition of $b$. Since
$$
\varphi = E * (-\Delta )^{1/2} \triangle^N\,\varphi \,,
$$
taking derivatives of both sides we obtain
\begin{equation*}\label{eq12}
P(\partial)\,\varphi = P(\partial)\, E * (-\Delta )^{1/2}
\triangle^N \, \varphi\,.
\end{equation*}
To compute $P(\partial)E$ we take the Fourier transform
$$
\widehat{P(\partial)E}(\xi)= P(i\xi)\, \hat{E}(\xi)=i
\frac{P(\xi)}{|\xi|^d}\,.
$$
On the other hand, as it is well known (\cite[p.~73]{St}),
$$
\widehat{p.v.\frac{P(x)}{|x|^{n+d}}}\,(\xi)= \gamma_d
\,\frac{P(\xi)}{|\xi|^d}\,.
$$
See \eqref{eq7} for the precise value of $\gamma_d$, which is not
important now. We conclude that, for some constant $a_d$ depending
on $d$,
$$
P(\partial)E = a_d \,\,p.v.\frac{P(x)}{|x|^{n+d}}\,.
$$
Thus
$$
P(\partial)\varphi = a_d \,\, p.v.\frac{P(x)}{|x|^{n+d}}*\, (-\Delta
)^{1/2} \triangle^N\,\varphi =
 a_d\,\,T(b)\,.
$$
The only thing left is the computation of $P(\partial)\,\varphi$. We
have, by Corollary 1 in \cite{MOV},
\begin{equation*}
\begin{split}
P(\partial)\,\varphi  & = a_d \,\, K(x)\,\chi_{\BC} +
P(\partial)(A_0+A_1\,|x|^2+...+ A_{d-1}\,|x|^{2d-2})(x)\,\chi_B(x)\\
& = a_d\,\, K(x)\,\chi_{\BC}\,,
\end{split}
\end{equation*}
where the last identity follows from the fact that, since $P$ is
harmonic,
\begin{equation}\label{eq13}
P(\partial)(|x|^{2j})= 0\,,\quad 1 \leq j \leq d-1.
\end{equation}
The identity \eqref{eq13} is a special case of a formula of Lyons
and Zumbrun \cite{LZ} which will be discussed in the next section
(see Lemma 2).

Once \eqref{eq10} is at our disposition  we get, for $f$ in some $
L^p(\Rn)\,,\, 1\leq p <\infty$\,,
\begin{equation*}
\begin{split}
T^1 f(0) & = - \int \chi_{\BC}(y)\,K(y)\, f(y) \,dy \\
 & = -\int T(b)(y)\,f(y)\,dy \\
 &= \int b(y) \, Tf(y)\,dy  \\
 &= \int_{2B} Tf(y) b(y) dy  \, + \int_{\Rn \setminus 2B} Tf(y) b(y) dy \\
 &= \int_{2B} Tf(y) (b(y) -b_{2B})dy  \, +b_{2B}\,\int_{2B} Tf(y) dy  \,+ \int_{\Rn \setminus 2B} Tf(y) b(y) dy \\
  & = I + II +III.
\end{split}
\end{equation*}
Notice that $b_{2B}$ is a dimensional constant, because of the
definition of $b$. In particular, it is independent of $f\,.$ Hence
$$
|II| \leq C\, M(Tf)(0)
$$
To estimate the local term $I$ we remark that $b\in BMO(\Rn)\,.$
This follows from the fact that $b$ is a sum on $j$ of the $j$-th
Riesz transform of a bounded function (depending on $j$). Hence we
can apply \eqref{John-Nirenberg-LLogL} to get
\begin{equation*}
\begin{split}
|I| & \le C\|b\|_{BMO}\|Tf\|_{L(\log L),2B} \\
   &\le C\,\|Tf\|_{L(\log L),2B}\\
   & \le C\,M^2(Tf)(0),
\end{split}
\end{equation*}
where we have used \eqref{iteration} in the last inequality.

To estimate the term $III$ we first prove the decay estimate
\begin{equation}\label{decay1}
|b(x)| \le \frac{C}{|x|^{n+1}}, \qquad |x| > 2.
\end{equation}
>From the decay of $b$ we obtain
$$
III \le C \int_{|x|>2}|Tf(x)| \frac{1}{|x|^{n+1}} dx\le CM(Tf)(0),
$$
using a standard argument which consists in estimating the integral
on the annuli $\{2^k \le |x| < 2^{k+1}  \}$. Let us prove
\eqref{decay1}. From the definition of $b$ we see that $b=b_{1}
+b_{2}$, where
\begin{equation*}
\begin{split}
b_{1} & =  c_{n}(1-n)  \sum_{j=1}^{n} R_{j} \left( \frac{x_{j}}{|x|^{n+1}} \chi_{B^{c}}(x) \right) \\
\text{and} &   \\
b_{2 } &= \sum_{j=1}^{n} R_{j} (a_{j})\\
\end{split}
\end{equation*}
each $a_{j}$ being a bounded function supported on $B$ with zero
integral (indeed, $a_j$ is odd).

If $|x|> 2$, then, since the kernel of $R_j$ is $\rho
\,x_{j}|x|^{-n-1}\,, $\,\,for some numerical constant $\rho$
depending on $n$,
\begin{equation}\label{eq15b}
\begin{split}
R_j(a_j)(x) &= \rho\, \int_{|y|< 1} \frac{x_j-y_j}{|x-y|^{n+1}}\,\,a_j(y)\,dy \\
& = \rho \,
\int_{|y|<1}(\frac{x_j-y_j}{|x-y|^{n+1}}-\frac{x_j}{|x|^{n+1}})\,\,a_j(y)\,dy\,.
\end{split}
\end{equation}
Thus
$$
|R_j(a_j)(x)| \leq \frac{C}{|x|^{n+1}}\,, \quad |x|> 2\,,
$$
and hence $b_2$ satisfies the decay estimate \eqref{decay1} with $b$
replaced by $b_2$. That this is also the case for $b_1$  was shown
in \cite{MV}. We repeat the argument here for completeness. One has
\begin{equation*}
\begin{split}
\sum_{j=1}^{n} R_j(\frac{\rho\;  y_j}{|y|^{n+1}}\,\chi_{\BC}\,(y)) &
=
 \sum_{j=1}^{n} R_j *  R_j  -  \sum_{j=1}^{n}
 R_j(\frac{\rho\; y_j}{|y|^{n+1}}\,\chi_{B}\,(y)) \\
 & = \delta_{0} -  \sum_{j=1}^{n}
 R_j(\frac{\rho\; y_j}{|y|^{n+1}}\,\chi_{B}\,(y))\,,
\end{split}
\end{equation*}
where $\delta_{0}$ is the Dirac delta at the origin. If $|x|>2$,
then
\begin{equation*}
\begin{split}
R_j(\frac{y_j}{|y|^{n+1}}\,\chi_{B}\,(y))(x) & = \rho\,
\lim_{\epsilon\rightarrow 0}
\int_{\epsilon<|y|< 1} \frac{x_j-y_j}{|x-y|^{n+1}}\,\,\frac{y_j}{|y|^{n+1}}\,dy \\
& = \rho\, \lim_{\epsilon\rightarrow 0} \int_{\epsilon<|y|< 1}
(\frac{x_j-y_j}{|x-y|^{n+1}}-\frac{x_j}{|x|^{n+1}})\,\,\frac{y_j}{|y|^{n+1}}\,dy\,.
\end{split}
\end{equation*}
Hence
$$
|R_j(\frac{y_j}{|y|^{n+1}}\,\chi_{B}\,(y))(x)| \leq
\frac{C}{|x|^{n+1}} \int_{|y|< 1} \frac{dy}{|y|^{n-1}} \leq
\frac{C}{|x|^{n+1}}\,,
$$
which completes the proof of \eqref{decay1}.

\section{Proof of the sufficient condition: the polynomial case}
This section does not differ substantially from its analogue for the
even case. Nevertheless, we will present the argument in detail for
the reader's sake, because it is technically sophisticated and we
would like to describe clearly the changes that have to be made.

Let us assume that $T$ is an odd polynomial operator. This amounts
to say that for some odd integer $2N+1$\,, \,$N \geq 0$\,, the
function $ |x|^{2N+1}\,\Omega(x)$ is a homogeneous polynomial of
degree $2N+1$\,. Such a polynomial may be written as \cite
[p.69]{St}
$$
|x|^{2N+1}\,\Omega(x)=  P_1(x)|x|^{2N}+ ...+
P_{2j+1}(x)|x|^{2N-2j}+...+ P_{2N+1}(x)\,,
$$
where $P_{2j+1}$ is a homogeneous harmonic polynomial of degree
$2j+1$\,,\, $0\leq j \leq N $\,. In other words, the expansion of
$\Omega(x)$ in spherical harmonics is
$$
\Omega(x)= P_1(x)+P_3(x)+...+P_{2N+1}(x),\quad |x|=1\,.
$$

As in the previous section, we want to obtain an expression for the
kernel $K(x)$ off the unit ball $B$. For this we need the
differential operator $Q(\partial)$ defined by the polynomial
\begin{equation*}
 Q(x)=  \gamma_1 \, P_1(x)|x|^{2N}+ ...+ \gamma_{2j+1}\,
P_{2j+1}(x)|x|^{2N-2j}+...+ \gamma_{2N+1}\,P_{2N+1}(x) \,.
\end{equation*}
If $E$ is the standard fundamental solution of $(-\Delta )^{1/2}
\Delta^N$\,, then
$$
Q(\partial)E = i \,  p.v. \,\,K(x)\,,
$$
which may be easily verified by taking the Fourier transform of both
sides ($K$ is the kernel of $T$).

Take now the function $\varphi$ of the previous section. We have
$\varphi = E * (-\Delta )^{1/2}  \triangle^N\,\varphi$ and thus
$$
Q(\partial)\varphi = Q(\partial)E *  (-\Delta )^{1/2}
\triangle^N\,\varphi =  p.v.\,\, K(x) * b =  T(b)\,,
$$
where $b$ is defined as $ i\;(-\Delta )^{1/2} \triangle^N\,\varphi$.
On the other hand, by Corollary 2 of \cite{MOV}
\begin{equation}\label{eq14}
Q(\partial)\,\varphi =   i\, K(x)\,\chi_{\BC} +
Q(\partial)(A_0+A_1\,|x|^2+...+ A_{2N}\,|x|^{4N})(x)\,\chi_B(x)\,.
\end{equation}
Contrary to what happened in the previous section, the term
$$
S(x) : = - Q(\partial)(A_0+A_1\,|x|^2+...+ A_{2N}\,|x|^{4N})(x)
$$
does not necessarily vanish, the reason being that now $Q$ does not
need to be harmonic.

Our goal is to find a function $\beta \in BMO (\Rn)$, satisfying the
decay estimate
\begin{equation}\label{eq15}
|\beta(x)| \leq \,\frac{C}{|x|^{n+1}}\,,\quad |x|\geq 2\,,
\end{equation}
and
\begin{equation}\label{eq16}
S(x) \chi_B(x) = T(\beta)(x)\,.
\end{equation}
By \eqref{eq14}, the definition of $S(x)$ and \eqref{eq16}, we then
get
\begin{equation}\label{eq17}
i\;K(x)\chi_{\BC}(x)=  T(b)(x)+ T(\beta)(x) = T(\gamma) (x)\,,
\end{equation}
where $ \gamma = b + \beta$ belongs to $BMO$ and satisfies the decay
estimate \eqref{eq15} with $\beta$ replaced by $\gamma$. Once this
is achieved the proof of $(i)$ is just the argument presented in
Section 4.

\medskip

To construct $\beta$ satisfying  \eqref{eq15} and \eqref{eq16} we
resort to our hypothesis, condition $(iii)$ in the Theorem, which
says that $T = R \circ U$, where  $U$ is an invertible operator in
the algebra $A$, $R$ is a higher order Riesz transform and the
polynomial $P$ which determines $R$ divides $P_{2j+1}$\,,\, $0\leq j
\leq N $\,, in the ring of polynomials in $n$ variables with real
coefficients. The construction of $\beta$ is performed in two steps.

The first step consists in proving that there exists a function
$\beta_1$ in $BMO$\,, satisfying some additional properties, such
that
\begin{equation}\label{eq18}
S(x) \chi_B(x) = R(\beta_1)(x)\,.
\end{equation}
It will become clear later what these {\it additional properties}
 are and how they are used. To prove \eqref{eq18} we need an explicit formula for
$S(x)$ and for that we will make use of the following formula of
Lyons and Zumbrun \cite{LZ}.

\begin{lemma} \label{L2}
Let $L$ be a homogeneous polynomial of degree $l$ and let $f$ be a
smooth function of one variable. Then
$$
L(\partial)f(r)= \sum_{\nu \geq 0} \frac{1}{2^\nu\,\nu !}\,\,
\Delta^\nu L(x)\,\left(\frac{1}{r}\frac{\partial}{\partial
r}\right)^{l-\nu} f(r)\,, \quad r=|x|\,.
$$
\end{lemma}

An immediate consequence of Lemma 1 is

\begin{lemma}\label{L3}
Let $P_{2j+1}$ a homogeneous harmonic polynomial of degree $2j+1$
and let $k$ be a non-negative integer. Then
$$
P_{2j+1}(\partial)(|x|^{2k}) = 2^{2j+1}\,\frac{k!}{(k-2j-1)!}
\,\,P_{2j+1}(x)\,|x|^{2(k-2j-1)}\quad \text{if}\quad 2j+1\leq k\,,
$$
and
$$
P_{2j+1}(\partial)(|x|^{2k}) = 0\,,\quad \quad  k < 2j +1 \,.
$$
\end{lemma}

On the other hand, a routine computation gives

\begin{equation}\label{eq19}
\triangle^{j}(|x|^{2k}) = 4^j\,\frac{j! \,k!
}{(k-j)!}\binom{\frac{n}{2}+k-1}{j} \,|x|^{2(k-j)}\,,\quad j \leq k
\,,
\end{equation}
and
\begin{equation}\label{eq20}
\triangle^{j}(|x|^{2k}) = 0 \,,\quad k < j \,.
\end{equation}
By Lemma 2, \eqref{eq19} and \eqref{eq20} we get that for some
constants $c_{k,j}$ one has, in view of the definitions of $Q(x)$
and $S(x)$,
\begin{equation}\label{eq21}
S(x)= \sum_{j=0}^{N-1} \sum_{k=0}^{N-1-j}
c_{k,j}\,\,P_{2j+1}(x)\,\,|x|^{2k}\,.
\end{equation}
Therefore it suffices to prove \eqref{eq18} with $S(x)$ replaced by
$P_{2j+1}(x)\,|x|^{2k}$\,,\,for $0 \leq j \leq N-1$ and each
non-negative integer $k\le N-1-j$. The idea is to look for an
appropriate function $\psi$ such that
\begin{equation}\label{eq22}
P(\partial)\psi (x) = P_{2j+1}(x)\,|x|^{2k}\,\chi_B(x)\,.
\end{equation}
Let $2d+1$ be the degree of $P$. Assume for the moment that
\eqref{eq22} holds and $\psi$ is good
 enough. Then
$$
\psi = E *    (-\Delta )^{1/2}\Delta^d \psi\,,
$$
where $E$
 is the fundamental solution of $ (-\Delta )^{1/2}\triangle^d$.  Hence
 $$
P(\partial)\psi = P(\partial)E *  (-\Delta )^{1/2} \Delta^d \psi\, =
c\, p.v.\,\frac{P(x)}{|x|^{n+2d+1}} *  (-\Delta )^{1/2}\Delta^d
\psi\,= R(\beta_1)\,,
$$
where $\beta_1$ is defined as $ c\,  (-\Delta )^{1/2}\Delta^d \psi
=\displaystyle  c\,\sum_{i=1}^{n} R_{i}\partial_{i} ( \Delta^d
\psi)$\,. The conclusion is that we have to solve \eqref{eq22} in
such a way that $\partial_{i} \Delta^d \psi$ is a bounded function
supported on $B$ with zero integral, $1 \le i \le n $. If this is
the case, then $\beta_{1}$ is in $BMO$ and satisfies the decay
estimate $|\beta_{1}(x)|\le C|x|^{-n-1}$ if $|x|
>2$, as we proved before (see \eqref{eq15b}).

Taking Fourier transforms in \eqref{eq22} we get
\begin{equation}\label{eq23}
i (-1)^d P(\xi)\, \widehat{\psi} (\xi) = i
(-1)^{j+k}\,P_{2j+1}(\partial)\,\triangle^k
\left(\widehat{\chi_B}(\xi)\right)\,.
\end{equation}
Recall that for $m=n/ 2$ one has \cite [B.5]{GrCF}
$$
\widehat{\chi_B}(\xi)= (2\pi)^{m}\,\,\frac{J_m(|\xi|)}{|\xi|^m}
\,,\quad \xi \in \Rn\,,
$$
where $J_m$ is the Bessel function of order $m$. Set
$$
G_\lambda(\xi)= \frac{J_\lambda(|\xi|)}{|\xi|^\lambda}\,,\quad \xi \in
\Rn \,, \quad \lambda>0\,.
$$
In computing the right hand side of \eqref{eq23} we apply Lemma 3 to
$L(x)= P_{2j+1}(x) \,|x|^{2k}$ and $f(r)= G_m(r)$ and we get
$$
P(\xi)\, \widehat{\psi} (\xi) = (2\pi)^{m}(-1)^{j+k+d} \,\sum_{\nu \geq 0}
\frac{(-1)^{\nu +1}}{2^\nu\,\nu !}\,\triangle^\nu \left(P_{2j
+1}(\xi)\,|\xi|^{2k}\right)\,G_{m+2j+1+2k-\nu}(|\xi|)\,,
$$
owing to the well known formula (e.g. \cite [B.2]{GrCF})
$$
\frac{1}{r}\frac{d}{dr}\,G_{\lambda}(r) = - G_{\lambda+1}(r)\,,\quad
r>0\,,\quad \lambda>0\,.
$$
Since $P_{2j+1}(\xi)$ is homogeneous of degree $2j+1$\,, $\nabla
P_{2j+1}(\xi)\cdot \xi = (2j +1)\,P_{2j+1}(\xi)\,,$ and hence one
may readily show by an inductive argument that
$$
 \triangle^\nu \left(P_{2j+1}(\xi)\,|\xi|^{2k}\right)=
 a_{jk\nu}\,P_{2j+1}(\xi)\,|\xi|^{2(k-\nu)}\,,
$$
for some constants $a_{jk\nu}$\,. Thus, for some other constants $
a_{jk\nu}$\,, we get
\begin{equation}\label{eq24}
P(\xi)\, \widehat{\psi} (\xi) = \sum_{\nu \geq 0}
a_{jk\nu}\,P_{2j+1}(\xi)\,|\xi|^{2(k-\nu)}\,G_{m+2j+1+2k-\nu}(\xi)\,.
\end{equation}
By hypothesis $P$ divides $P_{2j+1}$ in the ring of polynomials in
$n$ variables and so
$$
P_{2j+1}(\xi)= P(\xi)\,Q_{2j-2d}(\xi)\,,
$$
for some homogeneous polynomial $Q_{2j-2d}$ of degree $2j-2d$\,.
Cancelling out the factor $P(\xi)$ in \eqref{eq24}\, we conclude
that
\begin{equation*}
 \widehat{\psi} (\xi) = Q_{2j-2d}(\xi)\,\,\sum_{\nu = 0}^k
a_{jk\nu}\,|\xi|^{2(k-\nu)}\,G_{m+2j+1+2k-\nu}(|\xi|)\,.
\end{equation*}
Since \cite [B.5]{GrCF}
\begin{equation*}
\widehat{\left((1-|x|^2)^\lambda \,\chi_B(x)\right)}(\xi) =
c_\lambda\,
 G_{m+\lambda}(|\xi|)\,,
\end{equation*}
we finally obtain

\begin{equation*}
 \psi(x) = Q_{2j-2d}(\partial)\,\,\sum_{\nu = 0}^k
{a}_{jk\nu}\,\triangle^{k-\nu}\left( (1-|x|^2)
^{2j+1+2k-\nu}\,\,\chi_B(x)\right)\,,
\end{equation*}
for other constants ${a}_{jk\nu}\,.$  \,\,Observe that $\psi$
restricted to $B$ is a polynomial which vanishes
 on $\partial B$ up to order $2d+1$ and $\psi$ is zero off $B$. Therefore,  $
\partial_{i} \triangle^d \psi$ has zero integral, is supported on $B$ and its restriction to $B$
is a polynomial, $1 \le i \le n$. This completes the first step of
the construction of $\beta$.

The second step proceeds as follows. Since by hypothesis $T=R\circ U
$, with $U$ invertible in the algebra $A$, we have
$$
R(\beta_1)= R\circ U (U^{-1}\beta_1)= T(U^{-1}\beta_1)\,.
$$
Setting
\begin{equation}\label{eq24bis}
\beta=U^{-1}\beta_1\,,
\end{equation}
we are only left with the task of showing that
\begin{equation*}\label{eq25}
\beta \in BMO(\Rn)
\end{equation*}
and that, for some positive constant $C$\,,
\begin{equation}\label{eq26}
 | \beta (x)|  \leq \frac{C}{|x|^{n+1}}\,,\quad |x| \geq 2\,.
\end{equation}
Since $U^{-1} \in A$\,, for some real number $\lambda$ and some
smooth homogeneous Calder\'{o}n-Zygmund operator $V$\,,
$$
U^{-1} = \lambda\,I + V \,.
$$
Thus
$$
\beta = \lambda\,\beta_1 + V (\beta_1)\,.
$$
By construction, $\beta_{1}=\displaystyle \sum_{i=1}^{n}
R_{i}a_{i}$, where each $a_{i}$ is a bounded function supported on
$B$ with zero integral. By \eqref{eq15b} $\beta_1$ satisfies
\eqref{eq26} with $\beta$ replaced by $\beta_1\,.$  Clearly,
$\beta\in BMO$ and so we only have to get the decay estimate
\eqref{eq26} with $V(\beta_{1})$ in place of $\beta$. In fact,
$$
V(\beta_{1}) =  \sum_{i=1}^{n} V\, R_{i}a_{i}= \sum_{i=1}^{n}
\lambda_{i} a_{i} + \sum_{i=1}^{n} V_{i} a_{i},
$$
because each $V\, R_{i }\in A$ and thus  $V\, R_{i } = \lambda_{i} I
+V_{i}$ for some real number $\lambda_{i}$ and some smooth
homogeneous Calder\'{o}n-Zygmund operator $V_{i}$. Following the
argument we used in \eqref{eq15b} with $V_i$ in place of $R_i$ we
finally get \eqref{eq26}.

\section{Proof of the sufficient condition: the general case}
In \cite{MOV} several facts about the convergence of the expansion
\eqref{eq6} of $\Omega$ in spherical harmonics were established. In
particular, since $\Omega$ is infinitely differentiable on the unit
sphere and has the spherical harmonics expansion
\begin{equation}\label{eq33}
\Omega(x) =  \sum_{j \geq 0}^\infty \,P_{2j+1}(x)\,.
\end{equation}
one has that, for each positive integer $M$,
\begin{equation*}\label{eq32}
\sum_{j\geq 1} (2j+1)^M \,\|P_{2j+1}\|_{\infty} < \infty\,,
\end{equation*}
where the supremum norm is taken on $S^{n-1}$.

By hypothesis there is a homogeneous harmonic polynomial $P$ of
degree $2d+1$ such that $P_{2j+1} = P \;Q_{2j-2d}$, where
$Q_{2j-2d}$ is a homogeneous polynomial of degree $2j-2d$. As in
\cite{MOV}, one shows that the series $ \sum_{j} Q_{2j-2d}(x)$ is
convergent in $C^\infty(S^{n-1})$, that is, that for each positive
integer $M$
\begin{equation}\label{eq34}
\sum_{j\geq d} j^M \,\|Q_{2j-2d}\|_{\infty} < \infty\,.
\end{equation}

The scheme for the proof of the sufficient condition in the general
case is as in \cite{MOV}. Nevertheless, we will have to overcome
several new difficulties, which are not substantial but still
require significant work.

Taking a large partial sum of the series \eqref{eq33} we pass to a
polynomial operator $T_N$ (associated to a polynomial of degree
$2N+1$), which still satisfies the hypothesis~$(iii)$ of the
Theorem. Then we may apply the construction of Section~5 to $T_N$
and get functions $b_N$ and $\beta_N$. Unfortunately what was done
in Section~5 does not give any uniform estimate  in $N$, which is
precisely what we need to try a compactness argument. The rest of
the section is devoted to get the appropriate uniform estimates and
to describe the final compactness argument.

By hypothesis,  $T=R\circ U$, where $R$ is the higher order Riesz
transform associated to the harmonic polynomial~$P$ of degree~$2d+1$
that divides all $P_{2j+1}$, and $U$ is invertible in the
algebra~$A$. The Fourier multiplier of $T$ is
$$
\sum_{j=d}^\infty \gamma_{2j+1}\, \frac{P_{2j+1}(\xi)}{|\xi|^{2j+1}}
= \gamma_{2d+1}\,\frac{P(\xi)}{|\xi|^{2d+1}}\,\sum_{j\geq d}
\frac{\gamma_{2j+1}}
{\gamma_{2d+1}}\,\frac{Q_{2j-2d}(\xi)}{|\xi|^{2j-2d}},\quad \xi \in
\Rn \setminus \{0\}\,.
$$
Therefore the Fourier multiplier of $U$ is
\begin{equation}\label{eq36}
\mu(\xi) =  \gamma_{2d+1}^{-1}\, \sum_{j\geq d} \gamma_{2j+1}
\,\frac{Q_{2j-2d}(\xi)}{|\xi|^{2j-2d}} \,,
\end{equation}
and the series is convergent in $C^\infty(S^{n-1})$ because
$\gamma_{2j+1}\simeq (2j+1)^{-n/2}$ \cite[p.~226]{SW}. Set, for
$N\geq d$,
\begin{equation*}\label{eq37}
\mu_N(\xi) =  \gamma_{2d+1}^{-1}\, \sum_{j=d}^N \gamma_{2j+1}
\,\frac{Q_{2j-2d}(\xi)}{|\xi|^{2j-2d}},\quad \xi \in \Rn \setminus
\{0\}\,.
\end{equation*}
If
$$
K_N (x) = \sum_{j=d}^N \frac{P_{2j+1}(x)}{|x|^{2j+1+n}},\quad x \in
\Rn \setminus \{0\}\,,
$$
and $T_N$ is the polynomial operator with kernel $K_N$, then $T_N =
R \circ U_N$, where $U_N$~is the operator in the algebra~$A$ with
Fourier multiplier $\mu_N(\xi)$. From now on $N$ is assumed to be
big enough so that $\mu_N(\xi)$ does not vanish on $S^{n-1}$. In
fact, we will need later on the inequality
\begin{equation}\label{eq38}
 | \partial^\alpha \mu^{-1}_N(\xi) | \leq  C,\quad |\xi|= 1,\quad
 0 \leq |\alpha| \leq 2(n+3)\,,
\end{equation}
which may be taken for granted owing to the convergence in
$C^\infty(S^{n-1})$ of the series~\eqref{eq36}. In \eqref{eq38} $C$
is a positive constant depending only on the dimension~$n$ and on
$\mu$.

Notice that $T_N$ satisfies condition $(iii)$ in the Theorem (with
$T$ replaced by $T_N$), because $\mu_N(\xi) \neq 0$, $|\xi| = 1$,
and so we can apply the results of Section~5. In particular,
$$
\imath \;K_N (x) \chi_{\BC}(x) = T_N(b_N)(x) +  T_N( \beta_N)(x)\,,
$$
where $b_N$ and $\beta_N$ are respectively the functions $b$ and
$\beta$ appearing in \eqref{eq17}. It is important to remark that
$b_N$ does not depend on $T$. As \eqref{eq11bis} shows, the
function~$b_N$ depends on $N$ only through the fundamental solution
of the operator $(-\Delta)^{1/2}\triangle^N$. The uniform estimate
we need on $b_N$ is given by part~(i) of the next lemma. The
polynomial estimates in $N$ of (ii) and $(iii)$ are also central for
the compactness argument we are looking for, and they were not
present in the corresponding lemma for the even case (Lemma 8 in
\cite{MOV}).

\begin{lemma}\label{bN}
There exist a constant $C$ depending only on $n$ such that
\begin{enumerate}
\item[(i)]
$$
|\widehat{b_N}(\xi)| \leq C,\quad \xi \in \Rn\,,
$$
\item [(ii)]
$$
 \| b_N  \|_{BMO} \leq  C \,(2N+1)^{2n}\,,
$$
and
\item [(iii)]  $$
 \| b_N  \|_{2} \leq  C \,(2N+1)^{2n}\,,
$$
where $\| {\cdot} \|_{BMO}$ and $\|  {\cdot} \|_2$ denote
respectively the $BMO$ and $L^2$ norms on $\Rn.$
\end{enumerate}
\end{lemma}

\begin{proof}
We first prove (i). Let $h_1,\dotsc,h_d$ be an orthonormal basis of
the subspace of $L^2(d\sigma)$ consisting of all homogeneous
harmonic polynomials of degree $2N+1$. As in the proof of Lemma 6 in
\cite{MOV} we have $h_1^2 +\dotsb+h_d^2 = d$, on $S^{n-1}$. Set
$$
H_j(x)= \frac{1}{\gamma_{2N+1} \sqrt{d}}\,h_j(x), \quad x \in \Rn\,,
$$
and let $S_j$ be the higher order Riesz transform with kernel
$K_j(x) = H_j(x)/|x|^{2N+1+n}$. The Fourier multiplier of $S_j^2$ is
$$
\frac{1}{d}\,\frac{h_j(\xi)^2}{|\xi|^{4N+2}}, \quad 0 \neq \xi \in
\Rn\,,
$$
and thus
$$
\sum_{j=1}^d S_j^2 = I \,.
$$
By \eqref{eq10}, we get
$$
K_j(x)\,\chi_{\BC}(x)=  S_j(b_N)(x),\quad x \in \Rn,\quad 1 \leq j
\leq d\,,
$$
and so
\begin{equation}\label{eq39}
b_N = \sum_{j=1}^d S_j\left( K_j(x) \,\chi_{\BC}(x)\right)\,.
\end{equation}
We now appeal to a lemma of Calder\'{o}n and Zygmund (\cite{CZ}) and
we readily get $(i)$ (see \cite{MOV}).

We now turn to the proof of (ii) in Lemma~\ref{bN}. In view of the
expression~\eqref{eq39} for~$b_N$, we  obtain, by the standard
$L^\infty-BMO$ estimate,
\begin{equation*}\label{eq40}
\|b_N\|_{BMO} \leq C\, d\, \max_{1 \leq j \leq d} \|K_j\|_{CZ}\,
\|K_j\|_{L^\infty(\BC)} \,.
\end{equation*}
Recall that the constant of the kernel $K(x)= \Omega(x)/|x|^n$ of
the smooth homogeneous Calder\'{o}n-Zygmund operator $T$ is
\begin{equation*}\label{eq29}
\|T\|_{CZ} \equiv \|K\|_{CZ}= \|\Omega \|_\infty+ \| |x|\, \nabla
\Omega(x) \|_\infty\,.
\end{equation*}

As it is well known, $d \simeq (2N+1)^{n-2}$ \cite[p.~140]{SW}. On
the other hand
$$
\|K_j \|_{CZ} \leq  \|H_j\|_\infty+\|\nabla H_j\|_\infty\,,
$$
where the supremum norms are taken on $S^{n-1}$. Clearly
$$
\|H_j\|_\infty  = \frac{1}{\gamma_{2N+1}}
\|\frac{h_j}{\sqrt{d}}\|_\infty \leq \frac{1}{\gamma_{2N+1}} \simeq
(2N+1)^{n/2}\,.
$$
For the estimate of the gradient of $H_j$ we use the
inequality~\cite[p.~276]{St}
\begin{equation*}\label{eq40bis}
\| \nabla H_j\|_\infty \leq C\,(2N+1)^{n/2+1}\,\|H_j\|_2\,,
\end{equation*}
where the $L^2$ norm is taken with respect to $d\sigma$. Since the
$h_j$ are an orthonormal system,
$$
\|H_j\|_2 = \frac{1}{\sqrt{d}\,\gamma_{2N+1}} \simeq
\frac{(2N+1)^{n/2}}{(2N+1)^{(n-2)/2}} \simeq 2N+1\,.
$$
Gathering the above inequalities we get
$$
\|K_j\|_{CZ} \leq C\,(2N+1)^{n/2+2}\,.
$$
On the other hand, $ \|K_j\|_{L^\infty(\BC)}  \leq (2N+1)^{n/2}$ and
therefore
$$
\| b_N  \|_{BMO} \leq  C \,(2N+1)^{n-2}\,(2N+1)^{n/2+2} \,
(2N+1)^{n/2} = C\,(2N+1)^{2n}\,.
$$
The estimate (iii) in Lemma~\ref{bN} follows from  $\|b_N\|_2 \leq
C\, d\, \max_{1 \leq j \leq d} \|K_j\|_{CZ}\, \|K_j\|_{L^2(\BC)},$
$\|K_j\|_{L^2(\Rn \setminus B)} \le C\; \|H_j\|_\infty$ and the
previous inequalities.
\end{proof}

Our goal is now to show that under condition $(iii)$ of the Theorem
we can find a function $\gamma$ in $BMO(\Rn)$ such that
\begin{equation}\label{eq42}
\imath\;  K(x)\chi_{\BC}(x)=  T(\gamma)(x),\quad x \in \Rn\,,
\end{equation}
and having a decay as in \eqref{eq15} with $\beta$ replaced by
$\gamma$. If $T$ is a polynomial operator this was proven in the
preceding section for a $\gamma$ of the form $b+ \beta$ (see
\eqref{eq17}). The plan is to produce a different approach to this
result, which has the advantage that, when applied to $T_N$, gives a
uniform $BMO$ bound on $\gamma_N= b_N+ \beta_N$.

Since $\Omega$ has the expansion \eqref{eq33} in spherical
harmonics, we have
\begin{equation*}
\begin{split}
K(x)\chi_{\BC}(x)& = \sum_{j\geq 0}
\frac{P_{2j+1}(x)}{|x|^{2j+1+n}}\,\chi_{\BC}(x)\\*[5pt]
               & =  \sum_{j\geq 0} T_j(b_{j})(x)\,,
\end{split}
\end{equation*}
where $T_j$ is the higher order Riesz transform with kernel $
P_{2j+1}(x) / |x|^{2j+1+n}$ and $b_{j}$ is the function constructed
in Section~4 (see \eqref{eq10} and \eqref{eq11bis}). The Fourier
multiplier of $T_j$ is
$$
 \gamma_{2j+1}\, \frac{P_{2j+1}(\xi)}{|\xi|^{2j+1}} =
\gamma_{2d+1}\,\frac{P(\xi)}{|\xi|^{2d+1}}\, \frac{\gamma_{2j+1}}
{\gamma_{2d+1}}\,\frac{Q_{2j-2d}(\xi)}{|\xi|^{2j-2d}},\quad \xi \in
\Rn \setminus \{0\}\,.
$$
Let $S_j$ be the operator whose Fourier multiplier is
\begin{equation}\label{eq42bis}
\frac{\gamma_{2j+1}}
{\gamma_{2d+1}}\,\frac{Q_{2j-2d}(\xi)}{|\xi|^{2j-2d}},\quad \xi \in
\Rn \setminus \{0\}\,,
\end{equation}
so that $T_j = R \circ S_j $. Then
\begin{equation*}
\begin{split}
  K(x)\chi_{\BC}(x)& = \sum_{j\geq d} (R \circ S_j )(b_{j})\\*[5pt]
  &= \sum_{j\geq d} T \left((U^{-1} \circ S_j )(b_{j})\right)\\*[5pt]
& =  T\left(\sum_{j\geq d} (U^{-1}\circ S_j)(b_j)\right)\,.
\end{split}
\end{equation*}

The latest identity is justified by the absolute convergence of the
series \newline $\sum_{j\geq d} (U^{-1}\circ S_j)(b_j)$ in
$L^2(\Rn)$, which follows, using   (iii) in Lemma~\ref{bN}, from the
estimate
\begin{equation*}
\begin{split}
\sum_{j\geq d} \|(U^{-1}\circ S_j)(b_j)\|_2 & \leq C\, \sum_{j\geq
d} \|Q_{2j-2d}\|_\infty \,\|b_j\|_{L^2(\Rn)}\\*[4pt] & \leq C\,
\sum_{j\geq d} \|Q_{2j-2d}\|_\infty \, ( 2j+1)^{2n} < \infty\,.
\end{split}
\end{equation*}
We claim now that the series $\sum_{j\geq d} (U^{-1}\circ S_j)(b_j)$
converges  in $BMO(\Rn )$ to a function $-\imath\;\gamma$, which
will prove \eqref {eq42} . Observe that the operator $U^{-1}\circ
S_j \in A$ is not necessarily a Calder\'{o}n-Zygmund operator
because the integral on the sphere of its multiplier does not need
to vanish. However it can be written as $ U^{-1}\circ S_j = c_j I +
V_j$, where
$$
c_j = \frac{\gamma_{2j}} {\gamma_{2d}}\,\int_{S^{n-1}} \mu(\xi)^{-1}
\, Q_{2j-2d}(\xi)\,d\sigma(\xi)
$$
and $V_j$ is the Calder\'{o}n-Zygmund operator with multiplier
\begin{equation}\label{eq43}
\mu(\xi)^{-1} \frac{\gamma_{2j}}
{\gamma_{2d}}\,\frac{Q_{2j-2d}(\xi)}{|\xi|^{2j-2d}} - c_j \,.
\end{equation}
Now
$$
\sum_{j\geq d} (U^{-1}\circ S_j)(b_j) = \sum_{j\geq d} c_j\,b_j +
\sum_{j\geq d} V_j(b_j)
$$
and the first series offers no difficulties because, by
Lemma~\ref{bN} (ii) and \eqref{eq34}
$$
\sum_{j\geq d} |c_j|\, \|b_j\|_{BMO} \leq C\, \sum_{j\geq d}
(2j+1)^{-n/2} (2j+1)^{2n} \|Q_{2j-2d} \|_\infty < \infty\,.
$$
The second series is more difficult to treat.
 By Lemma 5 and Lemma~\ref{bN} (ii) and (iii),
\begin{equation*}
\begin{split}
  \|V_j(b_j) \|_{BMO} & \leq C\, \| V_j \|_{CZ} \;\|b_j\|_{BMO} \\*[3pt]
               & \leq C\, (2j+1)^{2n} \, \| V_j\|_{CZ} \,.
\end{split}
\end{equation*}
Estimating the Calder\'{o}n-Zygmund constant of the kernel of the
operator~$V_j$ is not an easy task, because we do not have an
explicit expression for the kernel. We do know, however, the
multiplier \eqref{eq43} of $V_j$. We need a way of estimating the
constant of the kernel in terms of the multiplier and this is what
Lemma 9 of \cite{MOV} achieves. The final outcome is
$$
\| V_j \|_{CZ} \leq C\,j^M \,\|P_{2j+1}\|_2\,,
$$
for some positive integer $M$ depending only on $n$ and the
polynomial $P.$ Thus
$$
  \|V_j(b_j) \|_{BMO} \leq C\,j^M \,\|P_{2j+1}\|_2\,,
$$
where again $M=M(n,P)$ is a positive integer. Hence the series
$\sum_{j\geq d} (U^{-1}\circ S_j)(b_j)$ converges in $BMO(\Rn)$ and
the proof of \eqref{eq42} is complete.

We are now ready for the discussion of the final compactness
argument that will complete the proof of the sufficient condition.
 We know from Section 5 (see \eqref{eq17})
that
\begin{equation}\label{eq45}
\imath \;K_N(x)\chi_{\BC}(x)=  T_N(b_N)(x)+ T_N( \beta_N)(x)\,.
\end{equation}
On the other hand, by the construction of the function $\gamma$ we
have just described, we also have
\begin{equation}\label{eq46}
\imath \; K_N(x)\chi_{\BC}(x)=  T_N(\gamma_N)(x), \quad \gamma_N =
\sum_{j\geq d}^N (U_N^{-1}\circ S_j)(b_j)\,.
\end{equation}
Notice that \eqref{eq38} guaranties that Lemma 9 of \cite{MOV} may
be applied to the operator $T_N$ and so the estimate of the $BMO$
norm of $\gamma_N$  is uniform in $N$. Since $T_N$ is injective,
\eqref{eq45} and \eqref{eq46} imply
\begin{equation}\label{eq46bis}
b_N + \beta_N = \gamma_N
\end{equation}
and, in particular, we conclude that the functions $b_N + \beta_N$
are uniformly bounded in $BMO(\Rn)$, a fact that cannot be derived
from the work done in Section~5.  On the other hand, Section 5 tells
us that $\gamma_N$ satisfies the decay estimate \eqref{eq15} with
$\beta$ replaced by $\gamma_N$, which we cannot infer from the
preceding construction of $\gamma$. The advantages of both
approaches will be combined now to get both the boundedness in $BMO$
and the decay property for $\gamma$.

In view of \eqref{eq46} and the expressions of the multipliers  of
$U_N$ and $S_j$ (see \eqref{eq42bis}),
$$
\widehat{\gamma_N}(\xi) = \sum_{j=d}^N
\frac{1}{\mu_N(\xi)}\,\frac{\gamma_{2j+1}}{\gamma_{2d+1}}\,\frac{Q_{2j-2d}(\xi)}{|\xi|^{2j-2d}}\,
 \widehat{b_j}(\xi)\,,
$$
which yields, by Lemma~\ref{bN} $(i)$ and \eqref{eq34} for $M=0$,
\begin{equation*}\label{eq48}
\begin{split}
  \|\widehat{\gamma_N} \|_{L^\infty(\Rn)} & \leq C\,\sum_{j= d}^N \|Q_{2j-2d}
  \|_{\infty}\\*[5pt]
  & \leq C\, \sum_{j= d}^\infty \|Q_{2j-2d}\|_{\infty} \\*[5pt]
  & \leq C\,,
\end{split}
\end{equation*}
where $C$ does not depend on $N$. Recall that, from \eqref{eq24bis}
in Section~5, we have
$$
\beta_N =U_N^{-1}(\beta_{1,N})\,,
$$
with $\beta_{1,N}= \displaystyle\sum_{i=1}^{n}R_{i}\partial_{i}
(f_{N})$, where $f_{N}$ is a $C^1$ function supported on $B$. Since
$$
\widehat{\beta_{1,N}} = \mu_N \, \widehat{\beta_{N}} = \mu_N
\,(\widehat{\gamma_{N}} - \widehat{b_{N}})\,,
$$
we have, again by Lemma~\ref{bN} $(i)$,
$$
 \|\widehat{\beta_{1,N}} \|_{L^\infty(\Rn)} \leq C\,.
$$
Therefore, passing to a subsequence, we may assume that, as $N$ goes
to $\infty$,
\begin{equation}\label{eq49bis}
\widehat{b_{N}}\longrightarrow a_0 \quad\quad\quad
{\text{and}}\quad\quad\quad \widehat{\beta_{1,N}} \longrightarrow
a_1\,,
\end{equation}
weak $*$ in $L^\infty(\Rn)$. Hence
$$
b_{N}\longrightarrow \Phi_0 = \mathcal{F}^{-1}{a_0} \quad\quad\quad
{\text{and}}\quad\quad\quad \beta_{1,N} \longrightarrow \Phi_1 =
 \mathcal{F}^{-1}{a_1}\,,
$$
in the weak $*$ topology of tempered distributions,
$\mathcal{F}^{-1}$ being the inverse Fourier transform.

We would like now to understand the convergence properties of the
sequence of the $\beta_N$'s . Since
$$
\widehat{\beta_{N}}(\xi) =  \mu_N^{-1}(\xi) \,
\widehat{\beta_{1,N}}(\xi)\,,
$$
and we have pointwise bounded convergence of $\mu_N^{-1}(\xi)$
towards $\mu^{-1}(\xi)$ on $\Rn \setminus\{0\}$, we get that
$\widehat{\beta_{N}} \rightarrow \mu^{-1}\,a_1$, in the weak $*$
topology of $L^\infty(\Rn)$. Thus $\beta_{N} \rightarrow
U^{-1}(\Phi_1)$ in the weak $*$ topology of tempered distributions.
Letting $N \rightarrow \infty$ in \eqref{eq46bis} we obtain
\begin{equation*}
\begin{split}
\gamma & = \Phi_0 + U^{-1}(\Phi_1)\\
           & = \Phi_0 +  \lambda \Phi_1 + V(\Phi_{1}) \, ,
\end{split}
\end{equation*}
where $\lambda$ is a real number and $V$ a smooth homogeneous
Calder\'{o}n-Zygmund operator.

We come now to the last delicate point of the proof, namely, that
one has the decay estimate
\begin{equation}\label{eq50}
|\gamma(x)| \leq \frac{C}{|x|^{n+1}},\quad |x|\geq 2\,.
\end{equation}
We claim that, as tempered distributions,
\begin{equation}\label{expre1}
\Phi_{0} =  \sum_{i=1}^n R_{i}\partial_{i}(S_{0}) +c \delta_{0}
\qquad \text{and} \qquad \Phi_{1} = \sum_{i=1}^n
R_{i}\partial_{i}(S_{1}) \, ,
\end{equation}
where $S_{0}$ and $S_{1}$ are distributions supported on $\overline
B$ and $c$ is a constant depending only on $n$. Recall that
$\beta_{1,N}= \displaystyle\sum_{i=1}^{n}R_{i}\partial_{i} (f_{N})$,
where $f_N$ is a $C^1$ function supported on $\overline{B}$, and, by
\eqref{eq11bis},
$$b_{N}=
\displaystyle\sum_{i=1}^{n}R_{i}\partial_{i}
\left(\frac{c_{n}}{|x|^{n-1}} \chi_{B^{c}}  + P_{N}   \chi_{B}
\right ) \,,$$ where $P_N$ is a polynomial. Set
$$
\alpha_N = \frac{c_{n}}{|x|^{n-1}} \chi_{B^{c}}  + P_{N}   \chi_{B}.
$$
Hence
$$\widehat{\beta_{1,N}}(\xi) =  |\xi |\, \widehat{f_{N}}(\xi)
\quad\quad\quad \text{and} \quad\quad\quad  \widehat{b_{N}}(\xi) =
|\xi|\, \widehat{\alpha_{N}}(\xi).
$$ By \eqref{eq49bis}, since $a_{0}, a_{1}\in L^{\infty}$ and
$\displaystyle{\frac{1}{|\xi|}}$ is locally integrable in $\Rn$
(because we may assume $n \geq 2$),
\begin{equation*}
\widehat{\alpha_{N}}\longrightarrow \frac{a_0}{ |\xi|}
\quad\quad\quad {\text{and}}\quad\quad\quad \widehat{f_{N}}
\longrightarrow \frac{a_1}{ |\xi|} \,,
\end{equation*}
 in the weak $*$ topology of tempered distributions. Hence
\begin{equation*}
\alpha_{N} \longrightarrow \alpha :=  \mathcal{F}^{-1}\left(
\frac{a_0}{ |\xi|} \right)\quad\quad\quad
{\text{and}}\quad\quad\quad f_{N}  \longrightarrow
S_{1}:=\mathcal{F}^{-1} \left(\frac{a_1}{ |\xi|}\right) \,,
\end{equation*}
in the weak $*$ topology of tempered distributions. Since each
$f_{N}$ is supported on $B$ we get that $S_{1}$ is also supported on
$\overline B$ and we obtain \eqref{expre1} for $\Phi_{1}$. On the
other hand, observe that
\begin{equation*}
P_{N}   \chi_{B} = \alpha_{N} - \frac{c_{n}}{|x|^{n-1}} \chi_{B^{c}}
\longrightarrow \alpha - \frac{c_{n}}{|x|^{n-1}} \chi_{B^{c}}:=
\alpha'\,,
\end{equation*}
with $\alpha'$ a tempered distribution supported on $\overline B$.
Set
$$S_{0}= \alpha' -
\frac{c_{n}}{|x|^{n-1}} \chi_{B}.$$  The claim now follows from the
chain of identities
\begin{equation*}
\begin{split}
\Phi_{0} & = \sum_{i=1}^n R_{i}\partial_{i}(\alpha)\\
               &=  \sum_{i=1}^n R_{i}\partial_{i}\left(\alpha' + \frac{c_{n}}{|x|^{n-1}}
               \chi_{B^{c}}\right) \\
               & =  \sum_{i=1}^n R_{i}\partial_{i}\left(\alpha' - \frac{c_{n}}{|x|^{n-1}} \chi_{B}\right) +
                \sum_{i=1}^n R_{i}\partial_{i}\left(\frac{c_{n}}{|x|^{n-1}}  \right)\\
                & = \sum_{i=1}^n R_{i}\partial_{i}( S_{0}) + c \sum_{i=1}^n R_{i}* R_{i}\\
                &=   \sum_{i=1}^n R_{i}\partial_{i}( S_{0}) + c\delta_{0} \, .
\end{split}
\end{equation*}

Therefore,
\begin{equation*}
\gamma = \sum_{i=1}^n R_{i}\partial_{i}(S_{0}) + c\delta_{0} +
\lambda \sum_{i=1}^n R_{i}\partial_{i}(S_{1}) + V\left( \sum_{i=1}^n
R_{i}\partial_{i}(S_{1}) \right) .
\end{equation*}
Write, for each $i$,\, $V\circ R_{i} = \lambda_i\,I + V_i$ for some
real number $\lambda_i$ and some homogeneous smooth
Calder\'{o}n-Zygmund operator $V_i$. Thus to get \eqref{eq50} it is
enough to show that
$$
|V(\partial_{i}S)(x)| \le \frac{C}{|x|^{n+1}}\,,\quad |x| \geq 2\,,
$$
where $V$ is a homogeneous smooth Calder\'{o}n-Zygmund operator and
$S$ a distribution supported on $\overline{B}$. Regularizing $S$ one
checks that, for a fixed $x$ with $|x|\geq 2$,
\begin{equation}\label{eq51}
\begin{split}
V(\partial_{i}S )(x) & = \langle \partial_{i}S,\, L(x-y)\rangle
\\*[3pt] & = - \langle S,\, \frac{\partial}{\partial y_i}\, L(x-y)\rangle \,,
\end{split}
\end{equation}
Since $ S$ is a distribution supported on~$\overline B$ there exists
a positive integer $\nu$ and a constant $C$ such that
\begin{equation}\label{eq52}
|\langle S ,\,\varphi\rangle | \leq C\sup_{|\alpha|\leq \nu}
\sup_{|y|\leq 3/2} |\partial^\alpha \varphi(y)|\,,
\end{equation}
for each infinitely differentiable function $\varphi$ on $\Rn$. The
kernel $L$ satisfies
$$
 |\frac{\partial^\alpha}{{\partial y}^\alpha}\,\frac{\partial}{{\partial
 y_i}}\,
L(x-y)| \leq \frac{C_\alpha}{|x|^{n+1+|\alpha|}},\quad |y| \leq
3/2\,,
$$
and hence, by \eqref{eq51} and \eqref{eq52},
$$
|V( \partial_{i} S)(x)| \leq \frac{C}{|x|^{n+1}},\quad |x|\geq 2\,,
$$
which proves \eqref{eq50} and then completes the proof of the
sufficient condition in the general case.

\section{Proof of the necessary condition }

The proof of the necessary condition is completely analogue to the
even case. We will just start the argument to help the reader in
capturing the context.

We first assume that $T$ is a polynomial operator with kernel
$$
K(x)=\frac{\Omega(x)}{|x|^n}=
\frac{P_1(x)}{|x|^{1+n}}+\frac{P_3(x)}{|x|^{3+n}}+...+\frac{P_{2N+1}(x)}{|x|^{2N+1+n}}\,,\quad
x \neq 0\,,
$$
where $P_{2j+1}$ is a homogeneous harmonic polynomial of degree
$2j+1$\,. Let $Q$ be the  homogeneous polynomial of degree $2N+1$
defined by
\begin{equation*}
 Q(x)=  \left(\gamma_1 \, P_1(x)|x|^{2N}+ ...+ \gamma_{2j+1}\,
P_{2j+1}(x)|x|^{2N-2j}+...+ \gamma_{2N+1}\,P_{2N+1}(x) \right)\,.
\end{equation*}
Then
$$
\widehat{p.v.K}(\xi) = \frac{Q(\xi)}{|\xi|^{2N+1}}\,.
$$
 Our assumption is now the $L^2(\Rn)$ control of $T^{*}f$ by $Tf$ (i.e., $(ii)$ in the statement of
the Theorem). Since the truncated operator $T^1$ at level $1$ is
obviously dominated by $T^{*}$, we have
$$
\int (T^1f)^2 (x)\,dx \leq \int (T^{*}f)^2 (x)\,dx \leq C\,\int
(Tf)^2 (x)\,dx \,.
$$
The kernel of $T^1$ is (see \eqref{eq14})
\begin{equation}\label{eq53}
K(x)\,\chi_{\BC}(x)= T(b)(x) + S(x)\,\chi_{B}(x)\,.
\end{equation}
where $b$ is given in equation \eqref{eq11bis} and
$$
-S(x)= Q(\partial)(A_0+A_1\,|x|^2+...+
A_{2N-1}\,|x|^{4N})(x)\,,\quad x \in \Rn \,.
$$
The reader may consult the beginning of Section 5 to review the
context of the definition of $S$\,. In view of \eqref{eq53} we have,
for each $f  \in L^2(\Rn)$,
\begin{equation*}
\begin{split}
  \|S\,\chi_{B} * f \|_2  & \leq C\, \|T^1 f \|_2 + \|b * Tf \|_2 \\
& \leq C\, (\|Tf \|_2 +\|\widehat{b} \|_\infty \|Tf \|_2) \\
& = C\,\|Tf\|_2 \,.
\end{split}
\end{equation*}
By Plancherel, the above $L^2$ inequality translates into a
pointwise inequality between the Fourier multipliers, namely,
\begin{equation}\label{eq54}
|\widehat{S\,\chi_B}(\xi)| \leq C\, |\widehat{p.v. K}(\xi)| = C\,
\frac{|Q(\xi)|}{|\xi|^{2N+1}}\,.
\end{equation}

Notice that $Q$ has plenty of zeros because it has zero integral on
the sphere. Our next aim is to use \eqref{eq54} to show that
$P_{2j+1}$ vanishes where $Q$ does.  For each function $f$ on $\Rn$
set $Z(f)= \{x \in \Rn : f(x)=0 \} \,.$

\begin{lemanom}[Zero Sets Lemma]
$$
Z(Q) \subset Z(P_{2j+1})\,, \quad 0 \leq j \leq N \,.
$$
\end{lemanom}
\begin{proof}
We know that $S$ has an expression of the form (see \eqref{eq21})
\begin{equation*}\label{eq54bis}
S(x)= \sum_{L=N+1}^{2N} \sum_{j=0}^{L-N-1}
c_{L,j}\,\,P_{2j+1}(x)\,\,|x|^{2(L-N-j-1)}\,.
\end{equation*}
Since $\widehat{\chi_B} = G_m (2\pi)^{m}\,,\,\,m=n/2$, Lemma \ref{L2} yields

\begin{equation}\label{eq54tris}
\begin{split}
\widehat{S\,\chi_{B}}(\xi) & = S(\imath\, \partial)\,
\widehat{\chi_B}(\xi)\\
& = \imath  (2\pi)^{n/2}\sum_{L=N+1}^{2N} \sum_{j=0}^{L-N-1}
c_{L,j}\,(-1)^{L-N}\,\,P_{2j+1}(\partial)\,\,\triangle^{L-N-j-1}\,G_{\n}(\xi)\\
 = (2\pi)^{n/2} & \sum_{L=N+1}^{2N} \sum_{j=0}^{L-N-1} \sum_{k=0}^{L-N-j-1}
c_{\,L,\,j,\,k}\,\,P_{2j+1}(\xi)\,\,|\xi|^{2(L-N-j-1-k)}
G_{\n+2(L-N)-1-k}(\xi)\,.
\end{split}
\end{equation}
The function $G_p(\xi)$ is, for each $p \geq 0$\,,\, a radial
function which is the restriction to the real positive axis of an
entire function \cite [B.6]{GrCF}. Set $\xi= r\,
\xi_0\,,\,\,|\xi_0|=1\,,\,\,r\geq0$\,. Then
\begin{equation}\label{eq55}
(2\pi)^{-n/2}\widehat{S \chi_B}(r\xi_0)= \sum_{p=0}^{\infty}
a_{2p+1}(\xi_0)\,r^{2p+1} \,,
\end{equation}
and the power series has infinite radius of convergence for each
$\xi_0$. Assume now that $Q(\xi_0) = 0$. Then, by \eqref{eq54},
$\widehat{S \chi_B}(r\xi_0)= 0\,$ for each $r \geq 0$\,, and hence
$a_{2p+1}(\xi_0) = 0\,, $\,for each $p \geq 0$\,. For $p=0$ one has
$a_{1}(\xi_0) = P_1(\xi_0)\,C_1$\,, where
$$
C_1 = \sum_{L=N+1}^{2N} c_{\,L,\,0,\,L-N-1}\,G_{\n+L-N}(0)\,.
$$
It will be shown later that $C_1 \neq 0$\,,\, and then we get
$P_1(\xi_0) =0 $\,. Let us make the inductive hypothesis that
$P_1(\xi_0)= ...=P_{2j-1}(\xi_0)=0$\,. Then we obtain, if $j \leq
N-1$\,,\, $a_{2j+1}(\xi_0) = P_{2j+1}(\xi_0)\,C_{2j+1}$\,,\,where

\begin{equation}\label{eq55bis}
C_{2j+1} = \sum_{L=N+1+j}^{2N}
c_{\,L,\,j,\,L-N-j-1}\,G_{\n+L-N+j}(0)\,.
\end{equation}

Since we will show that $C_{2j+1} \neq 0$\,,\,  $P_{2j+1}(\xi_0) =0
\,,\,\, 0 \leq j \leq N-1$\,. We have
$$
0= Q(\xi_0)=  \sum_{j=0}^N \gamma_{2j+1} \, P_{2j+1}(\xi_0)\,,
$$
and so we also get $P_{2N+1}(\xi_0)=0$\,. Therefore the zero sets
Lemma is completely proved provided we have at our disposition the
following formula, which in particular shows that $C_{2j+1} \neq
0\,,\,\,  0 \le j \le N-1 \,.$
\begin{lemma}\label{L4}
\begin{equation*}
\begin{split}
C_{2j+1}& = \frac{1}{2^{\n}} \; \frac{(-1)^{j}}{4^{j} (2j+1)
\Gamma(\n +2j +1)}, \quad  0 \le j \le N-1 \,.
\end{split}
\end{equation*}
\end{lemma}
The proof of Lemma \ref{L4} is lengthy and rather complicated from
the computational point of view, and so we postpone it to Section 8
\,.
\end{proof}

Notice that, although the constants $C_{2j+1}$ are non-zero, they
become rapidly small  as the index $j$ increases and they oscillate
around $0$.

The reason why Lemma \ref{L4} is involved is that one has to trace
back the exact values of the constants $C_{2j+1}$ from the very
beginning of our proof of \eqref{eq53}. This forces us to take into
account the exact values of various constants. For instance, those
which appear in the expression of the fundamental solution of
$(-\Delta)^{1/2} \triangle^N$ and the constants $A_0, A_1,...,
A_{2N}$ in formula \eqref{eq11}\,. Finally,  we need to prove some
new identities involving a triple sum of combinatorial numbers, in
the spirit of those that can be found in the book of R.\ Graham D.
Knuth and O.\ Patashnik \cite{GKP}\,.

The remaining of the proof of the necessary condition is basically a
plain translation of what was done in the even case. One first
completes the proof of the polynomial case by an appropriate
division process. Then the general case must be faced. We reduce to
the polynomial case by truncating the spherical harmonics expansion
of $\Omega$. Denoting by $S_N$ the analogue of $S$ at the truncated
level we set $\xi= r\,\xi_0$\,, with $|\xi_0| = 1$ and $ r>0$\,.
Rewrite \eqref{eq55} with $S$ replaced by $S_N$ and $a_{2p+1}$ by
$a_{2p+1}^N$\, :
\begin{equation*}
(2\pi)^{-n/2}\widehat{S_N \chi_B}(r\xi_0)= \sum_{p=0}^{\infty} a_{2p+1}^N
(\xi_0)\,r^{2p+1} \,.
\end{equation*}
As in the even case, it is a remarkable key fact that for a fixed
$p$ the sequence of the $a_{2p+1}^N$ stabilizes for N large. This
fact depends on a laborious computation of various constants and
will be proved in Section 8 in the following form.
\begin{lemma}\label{L5}
If $p+1 \leq N $\,, then \,\,$a_{2p+1}^N = a_{2p+1}^{p+1}$\,.
\end{lemma}
If $p\geq0$ and $p+1 \leq N $ we set $a_{2p+1}= a_{2p+1}^N$\,. We
need an estimate for the $a_{2p+1}^N$\,, which will be proved as
well in Section 8.

\begin{lemma}\label{L6}
We have, for a constant $C$ depending only on $n$\,,

\begin{equation}\label{eq59bis}
|a_{2p+1}| \leq \frac{C}{p!\,4^p} \,\sum_{j=0}^p \|P_{2j+1}\|_\infty
\,, \quad 0 \le p \le N-1\,,
\end{equation}

and

\begin{equation}\label{eq59tris}
|a_{2p+1}^N | \le  \frac{C}{4^p} \, \frac{ {N+\n -\frac{1}{2}\choose
N} }{ {N-\frac{1}{2}\choose N}} \, \sum_{j=0}^{N-1}
\|P_{2j+1}\|_\infty  \,, \quad 1 < N \le p \,\,.
\end{equation}

\end{lemma}

\section{Proof of the combinatorial Lemmata }

This section will be devoted to prove lemmas \ref{L4}, \ref{L5} and \ref{L6} stated and
used in the preceding sections. The arguments are parallel to those
of the even case, but many different computations have to be
performed . Owing to the intricate combinatorics involved we prefer
to write carefully down all calculations.

For the proof of Lemma \ref{L4} (see Section 7) we need to have
explicit expressions for the constants $C_{2j+1}$ and for this we
need to carefully trace back the path that led us to them. To begin
with we need a formula for the coefficients $A_L$ in \eqref{eq11}
and for that it is essential to have the expression for a
fundamental solution $E_N= E_N^n$ of $(-\Delta)^{1/2}\triangle^N$.
Recall that $\triangle^N( E_{N})(x) = c_{n}|x|^{1-n}$ in $\Rn$,
where the normalization constant $c_{n}$ is chosen so that
$\widehat{c_{n}/|x|^{n-1}}(\xi)=1/|\xi |$. One has

$$
E_N(x)= c_{n} |x|^{2N+1 -n}(\alpha(n,N)+\beta(n,N)\log|x|^2)\;,
$$
where $\alpha$ and $\beta$ are constants that depend on $n$ and
$N$\,.To write in close form $\alpha$ and $\beta$ we consider
different cases. Write $m=\displaystyle\frac{n-1}{2}$

Case 1: $n$ is even. Then
\begin{equation*}
\begin{split}
 \alpha(n,N)  & =\left(  \prod_{j=0}^{N-1} (2N+1-n-2j)\, (2N+1 -2(j+1)) \right)^{-1}\\
             & = \left(  {N -m \choose N} (2N)!\right)^{-1} \\
 \text{and} &   \\
 \beta(n;N) &=0.
\end{split}
\end{equation*}

Case 2: $n$ is odd  and $2N+1-n<0$. Then

\begin{equation*}
\begin{split}
 \alpha(n,N)  & =\left(  \prod_{j=0}^{N-1} (2N+1-n-2j)\, (2N+1 -2(j+1)) \right)^{-1}\\
             & = \left(\frac{(-2)^N (m-1)!}{(m-N-1)!}    \right)^{-1} \frac{2^{N-1}(N-1)!}{(2N-1)!}
              =\frac{(-1)^N (m-N-1)! (N-1)!}{2 (m-1)! (2N-1)!}\\
 \text{and} &   \\
 \beta(n;N) &=0.
\end{split}
\end{equation*}

Case 3:   $n$ is odd and $2N+1-n\ge 0$. Then

$$
\beta(n,N)=  \left( (-1)^{m+1} 2 (m-1)! \, (N-m)!\,
\frac{(2N-1)!}{(N-1)!}  \right)^{-1}
$$
and $\alpha(n,N)$ is a constant which we don't need to precise.

Recall that the constants   the constants $A_0, A_1,..., A_{2N}$ are
chosen so that the function (see \eqref{eq11})
\begin{equation*}
\varphi(x)= E(x)\,\chi_{\BC}(x) + (A_0+A_1\,|x|^2 +...+
A_{2N}\,|x|^{4N})\,\chi_B(x)\,,
\end{equation*}
and all its partial derivatives of order not greater than $2N$
extend continuously up to $\partial B$.

\begin{lemma}\label{L13} For $L=N+1, \dots, 2N$ we have
$$
A_{L}= c_{n}\, \frac{(-1)^{L+N}{L+m-N-1 \choose L-N}{N+m \choose 2N-L}}{
(2N)! {L \choose N}}
$$

\end{lemma}

\begin{proof}

Let $m= (n-1)/2$ and set $t=|x|^2$\,, so that

\begin{equation} \label{eq69}
E_N^n(x)\equiv E(t)=t^{N-m}(\alpha+\beta\log(t))
\end{equation}

Let $P(t)$ be the polynomial $\sum_{L=0}^{2N}A_Lt^L.$ By Corollary 2
in Section 2 we need that
$$
P^{k)}(1)=E^{k)}(1) \;,  \quad  0 \le k \le 2N\;.
$$

By Taylor's expansion we have that
$P(t)=\sum_{i=0}^{2N}\frac{E^{i)}(1)}{i!}(t-1)^{i},$ and hence, by
the binomial formula applied to $(t-1)^i$\,,
$$
A_L=\sum_{i=L}^{2N}\frac{E^{i)}(1)}{i!}(-1)^{i-L}{i \choose L}\;,
\quad  0 \le L \le 2N\;.
$$
Now we want to compute $E^{i)}(1)$ \,.\,  Clearly

$$
\left( \frac{d}{dt}\right)^i (t^{N-m})=(N-m ) \cdots
(N-m-i+1)t^{N-m-i}
$$
and it is zero when $m$ is integer and $i> N-m$. Notice that the
logarithmic term in \eqref{eq69} only appears when the dimension $n$
is odd (then $m$ is integer) and $N\ge m$. In this case, for each
$i\geq N+1$
$$
\left( \frac{d}{dt}\right)^i(t^{N-m}\log t )=
(N-m)!(-1)^{i-N+m-1}(i-N+m-1)!\;t^{-i+N-m}\,.
$$
Hence,  for $i\geq N+1$\;, we obtain
$$
\frac{E^{i)}(1)}{c_{n}}=\alpha(n,N)(N-m)\cdots(n-m-i+1)+
\beta(n,N)(N-m)!(-1)^{i-N+m-1}(i-N+m-1)!
$$
Consequently,
\begin{equation} \label{eq70}
\begin{split}
\frac{A_L}{c_{n}} & =(-1)^L\alpha(n,N)\sum_{i=L}^{2N}(N-m)\cdots
(N-m-i+1)\frac{(-1)^i}{i!}{i \choose L} + \\
& (-1)^{L-N+m-1}\beta(n,N)(N-m)! \sum_{i=L}^{2N}(i-N+m-1)!\frac{{i
\choose L}}{i!}.
\end{split}
\end{equation}
Let's remark that for the case $n$ odd and $N\geq m$ the first term
in \eqref{eq70} is zero, while for the cases $n$ even or $n$ odd and
$N < m$ the second term is zero because $\beta(n,N) = 0.$ This
explains why we compute below the two terms separately.

For the first term we show that
\begin{equation}\label{eq71}
\sum_{i=L}^{2N}(N-m)\cdots (N-m-i+1)\frac{(-1)^i}{i!} {i\choose L} =
(-1)^L {N-m \choose L}{m+N \choose 2N-L}
\end{equation}
Indeed, the left hand side of \eqref{eq71} is, setting $k=i-L$,

\begin{equation*}
\begin{split}
&
\frac{1}{L!}\sum_{k=0}^{2N-L}(N-m)\cdots(N-m-L-k+1)\frac{(-1)^{L+k}}{k!}
\\
& = (-1)^L{N-m \choose L} \sum_{k=0}^{2N -L}{m+L-N+k-1 \choose k}
\\
& = (-1)^L {N-m \choose L}{N+m  \choose 2N -L}\,,
\end{split}
\end{equation*}
where the last identity comes from (\cite [(5.9),p.159]{GKP})\,.

To compute the second term we first show that
\begin{equation}\label{eq72}
\sum_{i=L}^{2N}(i-N+m-1)!\frac{1}{i!}{i \choose L}=
\frac{(L-N+m-1)!}{L!}{N+m \choose 2N -L}\,.
\end{equation}
As before, setting $k=i-L$ and applying \cite [(5.9),p.159]{GKP}, we
see that the left hand side of \eqref{eq72} is

\begin{equation*}
\begin{split}
& \frac{1}{L!}\sum_{k=0}^{2N-L}(L+k-N+m-1)!\frac{1}{k!} \\
= & (L-N+m-1)! \sum_{k=0}^{2N -L}{m+L-N+k-1 \choose k} \\
= & \frac{(L-N+m-1)!}{L!}{N+m \choose 2N -L}\,. \\
\end{split}
\end{equation*}

We are now ready to complete the proof of the lemma distinguishing
$3$ cases.

\vspace{0.4 cm}

Case 1: $n$ even.

Since $\beta(n,N)=0$, replacing in \eqref{eq70} $\alpha(n,N)$ by its
value and using \eqref{eq71} we get, by elementary arithmetics,
\begin{equation*}
\begin{split}
\frac{A_L}{c_{n}}= &(-1)^L \frac{(-1)^L {N-m \choose L}{N+m  \choose 2N -L}}{2N {N -m \choose N} (2N-1)!} \\
= & (-1)^{L+N}\frac{{L+m-N-1 \choose L-N}{N+m \choose 2N-L}}{(2N)! {L \choose N}}\,.\\
\end{split}
\end{equation*}

\vspace{0.4 cm}

Case 2: $n$ is odd  and $2N+1-n<0$.

As in case 1 $\beta(n,N)=0$, and we proceed similarly using
\eqref{eq71} to obtain

\begin{equation*}
\begin{split}
\frac{A_L}{c_{n}}= & {N-m \choose L}{N+m  \choose 2N -L} \frac{(-1)^N (m-N-1)!
(N-1)!}{2 (m-1)! (2N-1)!}
\\[2mm]
= &(-1)^{L+N}\frac{{L+m-N-1 \choose L-N}{N+m \choose 2N-L}}{ (2N)! {L \choose N}} \, .\\
\end{split}
\end{equation*}

\vspace{0.4 cm}

Case 3: $n$ is odd and $2N+1-n\ge 0$.

Replacing in \eqref{eq70} $\alpha(n,N)$ and $\beta(n,N)$ by their
values and using \eqref{eq72} we get, by elementary arithmetics,
\begin{equation*}
\begin{split}
\frac{A_L}{c_{n}}= &(-1)^{L-N+m-1}\beta(n,N)(N-m)!  \frac{(L-N+m-1)!}{L!}{N+m \choose 2N -L} \\
= &
\frac{(-1)^{N+L} (N-1)! (L-N+m-1)!}{2 (m-1)! (2N-1)! L!}{N+m \choose 2N -L}\\[2mm]
= &(-1)^{L+N}\frac{{L+m-N-1 \choose L-N}{N+m \choose 2N-L}}{ (2N)! {L \choose N}} \, \\
\end{split}
\end{equation*}

\end{proof}

\begin{proof}[Proof of Lemma \ref{L4}]
 Recall that (see \eqref{eq55bis})
$$
C_{2j} = \sum_{L=N+1+j}^{2N} c_{\,L,\,j,\,L-N-j-1}\,G_{\n+L-N+j}(0)
\,.
$$
Thus, we have to compute the constants $c_{L,j,k}$ appearing in the
expression \eqref{eq54tris} for $\widehat{S\,\chi_{B}}(\xi)$\,. For
that we need the constants $c_{L,j}$ appearing in the formula
\eqref{eq21} for $S(x)$\,. We start by computing
$P_{2j+1}(\partial)\Delta^{N-j} (|x|^{2L})$\,. Using \eqref{eq19}
 and Lemma 4 one gets
\begin{equation*}
\begin{split}
P_{2j+1}(\partial)\Delta^{N-j} (|x|^{2L}) = \frac{ 2^{2N+1} L!
(N-j)!}{(L-N-j-1)!}{L-1+\n \choose N-j} P_{2j+1}(x) |x|^{2(L-N-j-1)}
\end{split}
\end{equation*}
if $L-N-j-1 \ge 0$ (and $=0$ if $L-N-j-1 < 0$)\,.

As in  \eqref{eq24} (Section 3), we express
$P_{2j+1}(\partial)\Delta^{L-N-j-1}G_{\n}(\xi)$ using Lemma \ref{L2}
applied to $f(r)=G_\n(r)$ and the homogeneous polynomial $L(x)=
P_{2j+1}(x)\,|x|^{2(L-N-j-1)}$\,. We obtain

\begin{equation*}
\begin{split}
P_{2j+1}(\partial)\Delta^{L-N-j-1}G_{\n}(\xi)  &  = \sum_{k\ge 0}\frac{1}{2^k k!}\Delta^{k} (  P_{2j+1}(x) | x|^{2(L-N-j-1}  )\,\left(\frac{1}{r}\frac{\partial}{\partial r} \right)^{2(L-N)-1-k} G_{\n} (\xi)\\
&  = \sum_{k\ge 0}\frac{(-1)^{k+1}}{2^k k!}\Delta^{k} (  P_{2j+1}(x) | x|^{2(L-N-j-1}  )\, G_{\n+2(L-N)-1-k} (\xi)\\
 & = \sum_{k=0}^{L-N-j-1} \frac{(-1)^{k+1}}{2^k k!}  4^{k}\frac{ (L-N-j-1) !}{(L-N-j-1-k)!} k! {\n+j+L-N-1 \choose k} \\
 & \quad \times  P_{2j+1}(\xi) | \xi|^{2(L-N-j-1-k)}  \, G_{\n+2(L-N)-1-k}(\xi) .\\
\end{split}
\end{equation*}

In view of the definitions of $Q(x)$ and $S(x)$,
\begin{equation*}
\begin{split}
S(x) & = -Q(\partial)\left (  \sum_{L=0}^{2N}A_{L}|x|^{2L} \right ) = -  \sum_{L=0}^{2N}A_{L} \sum_{j=0}^N \gamma_{2j+1} P_{2j+1 } (\partial) \Delta^{N-j}(|x|^{2L}) \\
& = - \sum_{L=N+1}^{2N} \sum_{j=0}^{L-N-1} A_{L} \gamma_{2j+1}
\frac{ 2^{2N+1} L!
(N-j)!}{(L-N-j-1)!}{L-1+\n \choose N-j} P_{2j+1}(x) |x|^{2(L-N-j-1)}\\
& = \sum_{L=N+1}^{2N} \sum_{j=0}^{L-N-1}  c_{L,j}P_{2j+1}(x)
|x|^{2(L-N-j-1)}\,,
\end{split}
\end{equation*}
where the last identity defines the $c_{L,j}$\,. In Section 5
\eqref{eq54tris} we set
\begin{equation*}
\begin{split}
\widehat{S\,\chi_{B}}(\xi) & = S(\imath\, \partial)\,
\widehat{\chi_B}(\xi)\\
& = \imath (2\pi)^{n/2} \sum_{L=N+1}^{2N} \sum_{j=0}^{L-N-1}
c_{L,j}\,(-1)^{L-N}\,\,P_{2j+1}(\partial)\,\,\triangle^{L-N-j-1}\,G_{\n}(\xi)\\
 = (2\pi)^{n/2}&\sum_{L=N+1}^{2N} \sum_{j=0}^{L-N-1} \sum_{k=0}^{L-N-j-1}
c_{\,L,\,j,\,k}\,\,P_{2j+1}(\xi)\,\,|\xi|^{2(L-N-j-1-k)}
G_{\n+2(L-N)-1-k}(\xi)\,.
\end{split}
\end{equation*}
Consequently,
\begin{equation*}
\begin{split}
c_{L,j,k} & =\imath c_{L,j} (-1)^{L-N}  \frac{(-1)^{k+1}}{2^k}  4^{k}\frac{ (L-N-j-1) !}{(L-N-j-1-k)!}  {\n+j+L-N-1 \choose k} \\[2mm]
& = \imath (-1)^{L+k+N}A_{L} \gamma_{2j+1} \frac{2^{2N+1} L!
(N-j)!}{(L-N-j-1-k)!}{L-1+\n \choose N-j} 2^{k}  {\n+j+L-N-1 \choose
k}\,.
\end{split}
\end{equation*}
Replacing $A_L$ by the formula given in lemma \ref{L13} and
performing some easy arithmetics we get

\begin{equation}\label{eq74}
\begin{split}
c_{L,j,k}&= \imath (-1)^k   c_{n}\gamma_{2j+1}  \frac{2^{2N+1} L! (N-j)!    {L-1+\n \choose N-j} 2^{k}  {\n+j+L-N-1 \choose k}}{(L-N-j-1-k)!}   \frac{{L+m-N-1 \choose L-N}{N+m \choose 2N-L}}{ (2N)! {L \choose N}}\\[2mm]
&=  \imath (-1)^k  c_{n} \gamma_{2j+1}  \frac{2^{k} (N-j)! (n-1) {L-1+\n
\choose N-j} {\n+j+L-N-1 \choose k} {N+\n -\frac{1}{2} \choose
N}}{(2N-L)! (L-N+\n -\frac{1}{2})  (L-N-j-1-k)! {N-\frac{1}{2}
\choose N}}\; .
\end{split}
\end{equation}

The final computation of the $C_{2j+1}$ is as follows.

\begin{equation*}
\begin{split}
C_{2j+1} &  =  \sum_{L=N+1+j}^{2N} c_{\,L,\,j,\,L-N-j-1}\,G_{\n+L-N+j}(0)     \\
& = \hspace{1 cm} \text{[by the explicit value of $G_p(0)$ given in \eqref{eq75} below]}\\
& =  \sum_{L=N+1+j}^{2N} c_{\,L,\,j,\,L-N-j-1} \frac{1}{2^{\n+L-N+j}\Gamma(\n+L-N+j+1)}\\[2mm]
& = \qquad \text{[by \eqref{eq74}]}\\
& =-\imath  c_{n} \sum_{L=N+1+j}^{2N} \frac{(-1)^{L-N-j} \gamma_{2j+1}
2^{L-N-j-1} (N-j)! (n-1)  {L-1+\n \choose N-j} {\n+j+L-N-1 \choose
L-N-j-1} {N+\n-1/2 \choose N}}{ (L-N+\n -\frac{1}{2}) (2N-L)!
{N-1/2 \choose N} 2^{\n+L-N+j}\Gamma(\n+L-N+j+1)} \\[2mm]
& =- \frac{\imath  c_{n}\gamma_{2j+1} (N-j)! (n-1) {N+\n-1/2 \choose N} } { 2^{\n+2j+1} {N-1/2 \choose N}  }\\
&\hspace{3 cm}  \sum_{L=N+1+j}^{2N} \frac{ (-1)^{L+N+j} {L-1+\n \choose N-j} {\n+j+L-N-1 \choose L-N-j-1} }{ (L-N+\n -\frac{1}{2}) (2N-L)! \Gamma(\n+L-N+j+1) } \\[2mm]
&=  \qquad  [\text{setting}\quad  L=i+N+j+1]  \\[2mm]
\end{split}
\end{equation*}
\begin{equation*}
\begin{split}
& =\frac{\imath \, c_{n}\gamma_{2j+1} (N-j)! (n-1) {N+\n-1/2 \choose N} } { 2^{\n+2j+1} {N-1/2 \choose N}  }\\
&\hspace{3 cm}  \sum_{i=0}^{N-j-1} \frac{ (-1)^{i} \displaystyle{N+i+j+\n \choose N-j} {\n+2j+i \choose i} }{ (i+j+\n +\frac{1}{2}) (N-j-i-1)! \Gamma(\n+i+2j+2) } \\[2mm]
&  = \qquad [\text{because }   \Gamma(\n+i+2j+2)= \Gamma(\n +2j) \prod_{k=0}^{i+1} (\n+2j+k) ]\\[2mm]
& =\frac{\imath\,  c_{n}\gamma_{2j+1} (N-j)! (n-1) {N+\n-1/2 \choose N} } { 2^{\n+2j+1} {N-1/2 \choose N} \Gamma(\n +2j) }\\
&\hspace{3 cm}  \sum_{i=0}^{N-j-1} \frac{ (-1)^{i} \displaystyle{N+i+j+\n \choose N-j} {\n+2j+i \choose i} }{ (i+j+\n +\frac{1}{2}) (N-j-i-1)!   \prod_{k=0}^{i+1} (\n+2j+k) } \\
& =  \qquad [\text{using Lemma \ref{sublema2} below }]\\
& =\frac{\imath \, c_{n}\gamma_{2j+1} (N-j)! (n-1) {N+\n-1/2 \choose N} } { 2^{\n+2j+1} {N-1/2 \choose N} \Gamma(\n +2j) } \frac{ 2 {N+1/2\choose N-j}}{(2N+1) (2j+\n)}\frac{\Gamma (\n +j+ 1/2)}{\Gamma (\n+ N +1/2)}\\
&= \qquad \text{[substituting  the value given in \eqref{eq7} in $\gamma_{2j+1}$]}\\
&=  c_{n}\left(\frac{\pi}{2}\right)^{\n}
\frac{(n-1)\Gamma(\frac{1}{2})}{\Gamma(\n +\frac{1}{2})}
\frac{(-1)^{j}}{4^{j} (2j+1) \Gamma(\n +2j +1)}\, \\
&= \qquad \text{[recalling the exact value of $c_{n}$]}\\
& = \frac{1}{2^{\n}} \; \frac{(-1)^{j}}{4^{j} (2j+1) \Gamma(\n +2j +1)} .
\end{split}
\end{equation*}

\end{proof}

\begin{lemma}\label{sublema2} For each $j=0, \dots , N-1$
$$
\sum_{i=0}^{N-j-1} \frac{ (-1)^{i} \displaystyle{N+i+j+\n \choose
N-j} {\n+2j+i \choose i} }{ (i+j+\n + \frac{1}{2}) (N-j-i-1)!
\prod_{k=0}^{i+1} (\n+2j+k) } =  \frac{ 2 {N+1/2\choose N-j}}{(2N+1)
(2j+\n)}\frac{\Gamma (\n +j+ \frac{1}{2})}{\Gamma (\n+ N
+\frac{1}{2})}
$$
\end{lemma}
\begin{proof} Denote the left hand side by $S$.
Using the identity $\Gamma(A)=\Gamma(A-k){A-1 \choose k} k!$, for
any non-negative integer $k$, and elementary arithmetics one gets
\begin{equation*}
\begin{split}
&\frac{\Gamma(\n+ N + \frac{1}{2})\displaystyle {N+i+j+\n \choose
N-j} {\n+2j+i \choose i} }{
\Gamma (\n + j+ \frac{1}{2}) (i+j+\n+\frac{1}{2})  \prod_{k=0}^{i+1} (\n+2j+k)\; } =\\[2mm]
= & {N+i+j+\n \choose N-j-1}{N+\n-1 \choose N-i-j-1}  {\n+i+j-1
\choose i} \frac{(N-i-j-1)!}{(2j+\n) (N-j)}\;,
\end{split}
\end{equation*}
and so
\begin{equation*}
\begin{split}
S &=  \frac{\Gamma (\n + j+ \frac{1}{2})}{\Gamma(\n+ N + \frac{1}{2}) (2j+\n) (N-j)} \\
&\hspace{2cm} \sum_{i=0}^{N-1-j} (-1)^{i} {N+i+j+\n \choose N-j-1}{N+\n-1 \choose N-i-j-1}{\n+i+j-\frac{1}{2} \choose i}\\[2mm]
& = \hspace{3cm} \left[ \text{because } {a+i\choose i}= (-1)^{i}{-a-1\choose i} \right]\\
&=  \frac{\Gamma (\n + j+ \frac{1}{2})}{\Gamma(\n+ N + \frac{1}{2}) (2j+\n) (N-j)} \\
&\hspace{2cm} \sum_{i=0}^{N-1-j} (-1)^{i} {N+i+j+\n \choose N-j-1}{N+\n-1 \choose N-i-j-1}{-\n-j-\frac{1}{2} \choose i}\\[2mm]
& =  [\text{by the triple-binomial identity (5.28) of  (\cite{GKP}, p. 171), see \eqref{eq74bis} below}]\\
& =   \frac{\Gamma (\n + j+ \frac{1}{2})}{\Gamma(\n+ N +
\frac{1}{2}) (2j+\n) (N-j)}   {N +\n+j \choose 0}
{N-\frac{1}{2} \choose N-j-1} \\[2mm]
& =   \frac{\Gamma (\n + j+ \frac{1}{2})}{\Gamma(\n+ N +
\frac{1}{2}) (2j+\n)} {N+\frac{1}{2} \choose N-j} \frac{2}{2N+1}\,.
\end{split}
\end{equation*}
For the reader's convenience and later reference we state the
triple-binomial identity (5.28) of \cite{GKP}\;:
\begin{equation}\label{eq74bis}
\sum _{k=0}^{n} {m-r+s\choose k}{n+r-s\choose n-k}{r+k\choose m+n} =
{r \choose m}{s \choose n}\qquad m,n\ge 0 \text{ integers}.
\end{equation}

\end{proof}
Our next task is to prove Lemma \ref{L5} and Lemma \ref{L6}. Setting
$\xi= r\,\xi_0$ in \eqref{eq54tris} we obtain

\begin{equation*}
\begin{split}
\frac{\widehat{S_{N}\,\chi_{B}}(r\xi_{0})}{(2\pi)^{n/2}} & = \sum_{L=N+1}^{2N}
\sum_{j=0}^{L-N-1} \sum_{k=0}^{L-N-j-1} c_{\,
L,\,j,\,k}\,\,P_{2j+1}(r\xi_{0})\,\,|r\xi_{0}|^{2(L-N-j-1-k)}
G_{\n+2(L-N)-1-k}(r\xi_{0})\\
& \hspace{0.8 cm} [\text{ make the change of indexes  $L=N+s$ and $|\xi_{0}|=1$} ]\\
& = \sum_{s=1}^{N} \sum_{j=0}^{s-1} \sum_{k=0}^{s-j-1}
c_{\,N+s,\,j,\,k}\,\,P_{2j+1}(\xi_{0})\,\,r^{2(s-k)-1} G_{\n+2s-1-k}(r)\\
& = \sum_{j=0}^{N-1} \sum_{s=j+1}^{N} \sum_{k=0}^{s-j-1}
c_{\,N+s,\,j,\,k}\,\,P_{2j+1}(\xi_{0})\,\,r^{2(s-k)-1} G_{\n+2s-1-k}(r)\\
& := \sum_{p=0}^{\infty} a_{2p+1}^{N}(\xi_{0}) r^{2p+1}.
\end{split}
\end{equation*}
In order to compute the coefficients $ a_{2p+1}^{N}(\xi_{0})$ we
substitute the power series expansion of $G_{q}(r)$ \cite
[B.2]{GrCF}, namely,
\begin{equation}\label{eq75}
G_{q}(r)=\sum _{i=0}^{\infty}\frac{(-1)^{i}}{i!\, \Gamma(q+i+1)}
\frac{r^{2i}}{2^{2i+q}}\,,
\end{equation}
in the last triple sum above.

\begin{proof}[Proof of Lemma \ref{L5}]
We are assuming that $0 \le p \le N-1$\,. It is crucial to remark
that, for this range of $p$, after introducing \eqref{eq75} in the
triple sum above,  only the values of the index $j$ satisfying $0
\le j \le p$ are involved in the expression for $a_{2p+1}^N$. Once
\eqref{eq75} has been introduced in the triple sum one should sum,
in principle, on the four indexes $i,j,s\; \text{and}\, k$\,. But
since we are looking at the coefficient of $r^{2p+1}$ we have the
relation $2(s-k)-1+2i=2p+1$\,,\, which actually leaves us with three
indexes. The range of each of these indexes is easy to determine and
one gets
$$
a_{2p+1}^N  = \sum_{j=0}^{p}P_{2j+1}(\xi_{0}) \sum_{i=0}^{p-j}
\sum_{s=p-i+1}^{N} c_{\,N+s,\,j,\,s-(p-i)-1}
 \times \text{coefficient of $r^{2i}$ from $G_{\n+s+p-i}(r)$}\,.
$$
In view of \eqref{eq75}

\begin{equation*}
\begin{split}
a_{2p+1}^N &= \sum_{j=0}^{p}P_{2j+1}(\xi_{0}) \sum_{i=0}^{p-j}
\sum_{s=p-i+1}^{N} c_{\,N+s,\,j,\,s-(p-i)-1}
\frac{(-1)^{i}}{i! 2^{i+\n+s+p}\Gamma(\n+s+p+1)}\\
& = \hspace{1cm}[\text{by the expression \eqref{eq74} for $c_{L,j,k}$ }]\\[2mm]
& =\imath\, c_{n}  \sum_{j=0}^{p}P_{2j+1}(\xi_{0}) \sum_{i=0}^{p-j}
\sum_{s=p-i+1}^{N}
\frac{(-1)^{i} (-1)^{s-(p-i)-1} \gamma_{2j+1}}{i! 2^{i+\n+s+p}\Gamma(\n+s+p+1)} \\[2mm]
& \qquad \frac{ \displaystyle   2^{s-(p-i)-1} (N-j)! (n-1) {N+s-1+\n \choose N-j} {\n+j+s-1 \choose s-(p-i)-1} {N+\n-\frac{1}{2} \choose N} } {(s+\n -\frac{1}{2})(N-s)!  (p-i-j)! {N-\frac{1}{2} \choose N}}\\[2mm]
& = (-1)^{p+1} c_{n} \frac{{N+\n-\frac{1}{2} \choose N}}{{N-\frac{1}{2} \choose N}}\sum_{j=0}^{p}P_{2j+1}(\xi_{0})  \frac{(-1)^{j} \pi^\n \Gamma(j+\frac{1}{2}) (N-j)!}{ \Gamma(\n+j +\frac{1}{2})2^{\n+2p+1}}\sum_{i=0}^{p-j} \frac{1 }{i! (p-i-j)! }\\
&\hspace{3cm}  \sum_{s=p-i+1}^{N} \frac{(-1)^s\displaystyle
{N+s-1+\n \choose N-j} {\n+j+s-1 \choose s-(p-i)-1} } {(s+\n-
\frac{1}{2})(N-s)! \Gamma(\n+s+p+1)}
\end{split}
\end{equation*}

In Lemma \ref{L15} below we give a useful compact form for the last
sum. Using it we obtain
\begin{equation*}
\begin{split}
\frac{a_{2p+1}^N}{c_{n}} & = (-1)^{p+1} \frac{{N+\n-\frac{1}{2} \choose N}}{{N-\frac{1}{2} \choose N}}\sum_{j=0}^{p}P_{2j+1}(\xi_{0})  \frac{(-1)^{j} \pi^\n \Gamma(j+\frac{1}{2}) (N-j)!}{ \Gamma(\n+j +\frac{1}{2})2^{\n+2p+1}}\sum_{i=0}^{p-j} \frac{1 }{i! (p-i-j)! }\\
&  \frac{ (-1)^{p+1-i}(N-p-1)! (p+1-j)! \Gamma(p+\n-i
+\frac{1}{2})}{(N-j)! \Gamma(N+\n +\frac{1}{2})  \Gamma(\n +
2p+2-i)} {N-\frac{1}{2} \choose N-p-1} {\n+2p-i+1 \choose p+1-j}
\end{split}
\end{equation*}
Easy arithmetics with binomial coefficients gives
$$
\frac{{N+\n-\frac{1}{2} \choose N} (N-p-1)!  {N-\frac{1}{2} \choose
N-p-1}} {{N-\frac{1}{2} \choose N}  \Gamma(N+\n +\frac{1}{2})} =
\frac{ \Gamma(\frac{1}{2})}{ \Gamma(\n +\frac{1}{2})
\Gamma(p+\frac{3}{2})  }
$$
We finally get the extremely surprising identity
\begin{equation}\label{eq78}
\begin{split}
\frac{a_{2p+1}^N}{c_{n}}  & = \frac{ \Gamma(\frac{1}{2}) (\frac{\pi}{2})^{\n} }{2^{2p+1} \Gamma(\n +\frac{1}{2})  \Gamma(p+\frac{3}{2})  }  \sum_{j=0}^{p} \frac{(-1)^j \Gamma(j+ \frac{1}{2})
P_{2j+1}(\xi_{0})} {\Gamma(\n+j+\frac{1}{2}) } \\
&\hspace{3cm} \sum_{i=0}^{p-j} \frac{(-1)^{i} \Gamma(\n +p-i
+\frac{1}{2})}{ i! (p-i-j)! \Gamma(\n +p -i +j+1 )}\; ,
\end{split}
\end{equation}
in which $N$ has miraculously disappeared.  Thus Lemma \ref{L5} is
proved.

\end{proof}

\begin{proof}[Proof of Lemma \ref{L6}]
We start by proving the inequality \eqref{eq59bis}, so that $0\le p
\le N-1$\,.  We roughly estimate  $a_{2p+1}= a_{2p+1}^N$ by putting
the absolute value inside the sums in \eqref{eq78}. The absolute
value of each term in the innermost sum in \eqref{eq78} is obviously
not greater than $1$ and there are at most $p+1$ terms. The factor
in front of $P_{2j+1}(\xi_0)$ is again not greater than $1$ in
absolute value. Denoting by $C$ the terms that depend only on $n$ we
obtain the desired inequality \eqref{eq59bis}\,.

We turn now to the proof of inequality \eqref{eq59tris}\,. Recall
that
\begin{equation*}
\begin{split}
& \frac{\widehat{S_{N}\,\chi_{B}}(r\xi_{0})}{(2\pi)^{n/2}}  =  \sum_{p=0}^{\infty} a_{2p+1}^{N}(\xi_{0}) r^{2p+1}\\
& =  \sum_{j=0}^{N-1} \sum_{s=j+1}^{N} \sum_{k=0}^{s-j-1}
c_{\,N+s,\,j,\,k}\,\,P_{2j+1}(\xi_{0})\,\,r^{2(s-k)-1}
G_{\n+2s-1-k}(r)\,.
\end{split}
\end{equation*}
Replacing $G_{\n+2s-1-k}(r)$ by the expression given by \eqref{eq75}
we obtain, as before, a sum with four indexes. Now we eliminate the
index $i$ of  \eqref{eq75} using $s-k+i=p+1$\,. Hence
\begin{equation*}
\begin{split}
a_{2p+1}^N & = \sum_{j=0}^{N-1}P_{2j+1}(\xi_{0}) \sum_{s=j+1}^{N}
\sum_{k=0}^{s-j-1} c_{\,N+s,\,j,\, k}
 \times \text{coefficient of $r^{2(p+1-s+k)}$ from $G_{\n+2s-1-k}(r)$}\\
&=\sum_{j=0}^{N-1}P_{2j+1}(\xi_{0}) \sum_{s=j+1}^{N} \sum_{k=0}^{s-j-1} c_{\,N+s,\,j,\, k}  \frac{(-1)^{p+1-s+k}}{(p+1-s+k)!\Gamma(\n+p+s+1)2^{2p+1+\n+k}}\\
&=\sum_{j=0}^{N-1}P_{2j+1}(\xi_{0}) \sum_{k=0}^{N-1-j} \sum_{s=j+k+1}^{N} c_{\,N+s,\,j,\, k}  \frac{(-1)^{p+1-s+k}}{(p+1-s+k)!\Gamma(\n+p+s+1)2^{2p+1+\n+k}}\\
&= c_{n}(-1)^{p+1} \frac{ {N+\n-\frac{1}{2} \choose N} \pi^{\n} (n-1) }{
4^{p} {N-\frac{1}{2} \choose N} \,
2^{\n+1} }\sum_{j=0}^{N-1}\frac{(-1)^{j} \Gamma(j+\frac{1}{2}) }{\Gamma( \n+ j+\frac{1}{2}) } P_{2j+1}(\xi_{0}) \sum_{k=0}^{N-1-j} \sum_{s=j+k+1}^{N} \\
&\qquad  \frac{(-1)^s (N-j)! \displaystyle{N+s-1+\n\choose
N-j}{\n+j+s-1\choose k}}{(p+1-s+k)! \Gamma(\n+p+s+1)
(s+\n-\frac{1}{2}) (N-s)! (s-j-1-k)!}
\end{split}
\end{equation*}
The second identity is just \eqref{eq75}\,. The third is a change of
the order of summation and the latest follows from the formula
\eqref{eq74} for the constants $c_{l,j,k}$ and some simplifications.

In view of the elementary fact that
\begin{equation*}
\begin{split}
(N-j)! \displaystyle{N+s-1+\n\choose N-j}{\n+j+s-1\choose k} =
\frac{\Gamma(s+\n+N)}{k! \Gamma(s+\n+j-k)}
\end{split}
\end{equation*}
we get

\begin{equation*}
\begin{split}
&\left| \sum_{k=0}^{N-1-j} \sum_{s=j+k+1}^{N}  \frac{(-1)^s (N-j)!
\displaystyle{N+s-1+\n\choose N-j}{\n+j+s-1\choose k}}{(p+1-s+k)!
\Gamma(\n+p+s+1) (s+\n-\frac{1}{2}) (N-s)! (s-j-1-k)!} \right | \le\\
&\le \sum_{k=0}^{N-1-j}
\frac{1}{k!}\sum_{s=j+k+1}^{N}\frac{1}{\Gamma(s+\n+j-k)(p+1-s+k)!
(\n+p+s) (s+\n-\frac{1}{2}) (N-s)! (s-j-1-k)!}\\
&\le \sum_{k=0}^{N-1-j} \frac{1}{k!}\sum_{s=j+k+1}^{N}\frac{1}{
(s-j-1-k)!} \le e^2 \;,
\end{split}
\end{equation*}
where in the first inequality we used that, since $N \le p$\,,
$$
\frac{\Gamma(s+\n+N)}{\Gamma(\n+p+s+1)} \le \frac{1}{\n+p+s}\;.
$$
The proof of \eqref{eq59tris} is complete.
\end{proof}

\begin{lemma}\label{L15}
Let $N-1\ge p\ge j +i \ge 0$ be non-negative integers and set
$m=p+1-i$. Then
\begin{equation*}
\begin{split}
& \sum_{s=0}^{N-m}  \frac{(-1)^s\displaystyle{\n+N+m+s-1 \choose
N-j}{\n+j+m+s-1
\choose s}}{(m+s+\n-\frac{1}{2})(N-m-s)!\Gamma(\n+2m+i+s)} \\[2mm]
= & \frac{(N-m-i)!\, (m+i-j)!\, \Gamma(m+\n-\frac{1}{2})}{(N-j)!\,
\Gamma(\n+2m +i)\, \Gamma(N+\n+\frac{1}{2})} {N-\frac{1}{2} \choose
N-m-i} {\n+2m+i-1 \choose m+i-j}.
\end{split}
\end{equation*}

\end{lemma}

\begin{proof} Denote the left hand side by $S$. Using repeatedly the identity $\Gamma(x+1)=x\,\Gamma(x)$ and arithmetics with binomial coefficients we have
\begin{equation*}
\begin{split}
& \frac{\displaystyle{\n+N+m+s-1 \choose N-j}{\n+j+m+s-1\choose s}}{\Gamma(\n+2m+i+s)} \\[2mm]
= & \frac{(N-m-i)!\, (m+i-j)! }{(N-j)!\, s! \, \Gamma(\n+2m+i) }
{\n+N+m+s-1 \choose N-m-i}{\n +2m+i-1 \choose m+i-j}
\end{split}
\end{equation*}
Then
\begin{equation*}
\begin{split}
S &= \frac{(N-m-i)!\, (m+i-j)! }{(N-j)!\, \Gamma(\n+2m+i) } {\n
+2m+i-1 \choose m+i-j} \sum_{s=0}^{N-m}
\frac{ (-1)^{s} {\n+N+m+s-1 \choose N-m-i}}{ s! (N-m-s)! (m+s+\n- \frac{1}{2}) }\\
 &=  \frac{(N-m-i)!\, (m+i-j)! }{(N-j)!\, \Gamma(\n+2m+i) } {\n +2m+i-1 \choose m+i-j} D(m,i) \; ,
\end{split}
\end{equation*}
where the last identity defines $D(m,i)$\,.  The only task left is
the computation of the sum $D(m,i)$. The identity
$$
\frac{1}{m+s+\n-\frac{1}{2}}=\frac{1}{m+\n-\frac{1}{2}}\left(1-\frac{s}{m+s+\n-\frac{1}{2}}\right),
$$
yields the expression
\begin{equation*}
\begin{split}
D(m,i)&
=\frac{1}{(m+\n -\frac{1}{2}) (N-m)!}\sum_{s=0}^{N-m} (-1)^s {N-m \choose s}{\n+N+m+s-1 \choose N-m-i} \\[2mm]
&- \sum_{s=1}^{N-m} \frac{ (-1)^{s} \displaystyle{\n+N+m+s-1 \choose
N-m-i}}{ s! (N-m-s)!  }
\frac{s}{(m+\n- \frac{1}{2})(m+s+\n- \frac{1}{2}) }\;.\\
\end{split}
\end{equation*}
The first sum in the above expression for $D(m,i)$ turns out to
vanish for $i\ge 1$. This is because
\begin{equation*}
\begin{split}
 & \sum_{s=0}^{N-m} (-1)^s {N-m \choose s}{\n+N+m+s-1 \choose N-m-i}  \\
 &= (-1)^{N-m} {N+m+\n -1 \choose -i} = 0 \;,
\end{split}
\end{equation*}
where the first identity follows from \cite [(5.24), p.169]{GKP} and
the second from the fact that ${m \choose n} = 0 $ if $n$ is a
negative integer. Hence, setting $r=s-1$,

\begin{equation*}
\begin{split}
& D(m,i) = -   \sum_{s=1}^{N-m} \frac{ (-1)^{s}
\displaystyle{\n+N+m+s-1 \choose N-m-i}}{ s! (N-m-s)!  }
\frac{s}{(m+\n- \frac{1}{2})(m+s+\n- \frac{1}{2}) }\\[2mm]
 & =\frac{1}{(m+\n- \frac{1}{2}) }      \sum_{r=0}^{N-(m+1)} \frac{ (-1)^{r} \displaystyle{\n+N+(m+1)+r-1 \choose N-(m+1)-(i-1)}}{ r! (N-(m+1)-r)!   ((m+1)+r+\n- \frac{1}{2})} \\
& = \frac{1}{(m+\n- \frac{1}{2}) }\;  D(m+1,i-1)\;.\\
\end{split}
\end{equation*}
Repeating  the above argument $i$ times we obtain that
$$
D(m,i)=\frac{1}{(m+\n -\frac{1}{2})(m+\n+\frac{1}{2})\cdots
(m+\n+i-\frac{3}{2})}\;D(m+i,0).
$$
To compute $D(m+i,0)$ or $D(p+1,0)$ we use the elementary identity
$$
 \frac{ \Gamma( N+\n+\frac{1}{2}) }{\Gamma(p+\n +\frac{1}{2}) s! (N-p-1-s)! (p+s+\n+ \frac{1}{2}) }= {N+\n-\frac{1}{2}\choose N-p-1-s}{p+s+\n-\frac{1}{2}\choose s} ,
$$
from which we get
\begin{equation*}
\begin{split}
D(p+1,0)&=\sum_{s=0}^{N-p-1} \frac{ (-1)^{s} {\n+N+p+s \choose N-p-1}}{ s! (N-p-1-s)! (p+s+\n+ \frac{1}{2}) }   \\
& =\frac{\Gamma(p+\n+\frac{1}{2})}{\Gamma(N+\n+\frac{1}{2})}\sum_{s=0}^{N-p-1} (-1)^s {\n+N+p+s \choose N-p-1} {N+\n-\frac{1}{2}\choose N-p-1-s}{p+s+\n-\frac{1}{2}\choose s}\\
& =\frac{\Gamma(p+\n+\frac{1}{2})}{\Gamma(N+\n+\frac{1}{2})}\sum_{s=0}^{N-p-1} {\n+N+p+s \choose N-p-1} {N+\n-\frac{1}{2}\choose N-p-1-s}{-p-\n-\frac{1}{2}\choose s}\\
&= \frac{\Gamma(p+\n+\frac{1}{2})}{\Gamma(N+\n+\frac{1}{2})}
{N-\frac{1}{2}\choose N-p-1}  ,
\end{split}
\end{equation*}
where in the third identity we applied \cite[(5.14), p. 164]{GKP}
and the latest equality is consequence of the triple-binomial
identity \eqref{eq74bis} \cite[(5.28), p.171]{GKP} (for $k=s$,  $
n=N-p-1, \, m=0, \, r=N+p+\n$ and $s=N-\frac{1}{2})\,.$
Consequently,
\begin{equation*}
\begin{split}
 D(m,i)& =  \frac{\Gamma(m+\n+i -\frac{1}{2})}{\Gamma(N+\n+\frac{1}{2})} {N-\frac{1}{2}\choose N-m-i} \frac{1}{(m+\n -\frac{1}{2})\cdots (m+\n+i-\frac{3}{2})}\\
&=\frac{\Gamma(m+\n-\frac{1}{2})}{\Gamma(N+\n+\frac{1}{2})} {N-\frac{1}{2}\choose N-m-i}   \, ,\\
\end{split}
\end{equation*}
which completes the proof of the lemma\,.
\end{proof}

\section{Failure of the pointwise estimate \eqref{pointwiseH1}} \label{contraejemploHilbert}

In this section we give a proof that \eqref{pointwiseH1} is false,
more direct than the one in \cite{MV}, and we show the connection
with the algebra of operators already mentioned.

\begin{teor} \label{counterHilbert}
 The following pointwise inequality is false for functions in
$L^2(\R)$:
\begin{equation}\label{pointwiseH}
H^*f(x)\le C\, M(Hf)(x)\,, \quad x \in \R\,.
\end{equation}
\end{teor}
{\bf{Remark.}} Notice that the Theorem implies that there is no
good-lambda inequality between $H^*(f)$ and $H(f)$.

Replacing $f$ by $H(f)$ in \eqref{pointwiseH} and recalling that
$H(Hf)=-f\,,\,\, f \in L^2(\R)$\,, we see that \eqref{pointwiseH} is
equivalent to saying that
$$H^*(H(f))(x)\le C\, M(f)(x)\,, \quad x \in \R\,,$$
for any $f\in L^2(\R)$\,.

\begin{lemma}
The operator $f \rightarrow H^*(Hf)$  fails to be of weak type
$(1,1)$.
\end{lemma}

\begin{proof}
To prove the Lemma it is enough to show that if $f=\chi_{(0,1)}$,
then there are positive constants $m$ and $C$ such that whenever
$x>m$,
\begin{equation}\label{log}
 H^*(Hf)(x)\ge C\,\frac{\log x}{x}
\end{equation}
Indeed, choosing $m>e$ if necessary, we have
$$ \sup_{\lambda>0} \lambda \, |\{ x\in \R: H^*(Hf)(x) >
\lambda \} | \ge \sup_{\lambda>0} \lambda \, |\{ x > m: \frac{\log
x}{x}
> C^{-1}\, \lambda \} |
$$
$$
= C\, \sup_{\lambda>0} \lambda \, | \{ x >m  :  \frac{\log x }{x}
> \lambda \} | \ge C\, \sup_{\lambda>0} \lambda \,(
\varphi^{-1}(\lambda) - m),
$$
where $\varphi$ is the decreasing function $\varphi :(e,\infty)
\rightarrow  (0,e^{-1})$,  given by $\varphi(x) = \frac{\log x
}{x}$. To conclude observe that the right hand side of the
estimate is unbounded as $\lambda \rightarrow 0$:
$$
\lim_{\lambda \rightarrow 0}      \lambda \varphi^{-1}(\lambda) =
\lim_{\lambda \rightarrow \infty} \varphi(\lambda) \lambda = \infty.
$$

To prove \eqref{log} we recall that for $f=\chi_{(0,1)}$
$$Hf(y)
=\log\frac{|y|}{|y-1|}.
$$

Let $m>1$ big enough that will be chosen soon. Take $x>m$.  Hence,
by definition,
$$H^*(Hf)(x)
\ge   \left| \int_{ |y-x|> m+x } \frac{1}{y-x}\,
\log\frac{|y|}{|y-1|} \,dy \right|
$$
and splitting the integral in the obvious way
$$\int_{-\infty}^{-m} \frac{1}{y-x}\log\frac{-y}{-y+1} \,dy +
\int_{2x+m}^{\infty} \frac{1}{y-x}\log\frac{y}{y-1}\,dy
$$
$$=
\int_{m}^{\infty} \frac{1}{x+y}\log\frac{y+1}{y} \,dy +
\int_{2x+m}^{\infty} \frac{1}{y-x}\log\frac{y}{y-1}\,dy =A(x)+B(x),
$$
where  both $A(x), B(x)$ are positive.  Hence
$$H^*(Hf)(x)
\ge A(x).
$$
Since $$\log(1+\frac1y) \approx \frac1y $$ as $y \rightarrow
\infty$, there is a constant $m>1$ such that whenever $y>m$
$$ \frac12<\frac{\log(1+\frac1y) }{\frac1y} < \frac32.$$
Hence, with this constant $m$ we have
$$ A(x) = \int_{m}^{\infty} \frac{1}{x+y}\log\left(1+\frac{1}{y}
\right)\,dy \approx \int_{m}^{\infty} \frac{1}{x+y} \,\frac{dy}{y}
= \frac1x \log \frac{y}{x+y}\Big|_{m}^{\infty} \approx \frac{\log
x}{x}\,,
$$
which proves \eqref{log}.

Notice that $B$ is better behaved :
$$ B(x) \le  \int_{2x+m}^{\infty} \frac{1}{y-x}\log\frac{y}{y-1} \,dy
\le \int_{2x+m}^{\infty} \frac2y \,\frac{dy}{y}
 \le \frac1x.
$$
\end{proof}

\section{Composition of operators : positive results } \label{compostionestimates}

We first discuss a proof of \eqref{fslogl} in Lemma \ref{pointwise}
using standard arguments except for a point that will be supplied.
We mention that in \cite{Le1} there is a different argument.

 Let $x\in \Rn$ and let
$Q=Q(x,r)$ be an arbitrary cube centered at $x$ and sidelength $r$.
It is enough to show that there exists $C>0$ such that for some
constant $c=c_{Q}$
\begin{equation}\label{1lemma}
\frac{1}{|Q|} \int_{Q} |Tf(y)- c| \,dy  \le C\, Mf(x).
\end{equation}

Let $f = f_{1}+f_{2}$, where $f_{1} = f\,  \chi_{2Q}$. If we pick
$c =   (T (f_{2}) ) _{Q} $ , we can estimate the left hand side of
\eqref{1lemma} by a multiple of
$$ \frac{1}{|Q|} \int_{Q} |T(f_1)(y)| \,dy
+ \frac{1}{|Q|} \int_{Q}  |T(f_{2}) - ( T(f_{2}) )_{Q}| \,dy = I +
II.
$$
To take care of II we use the regularity of the kernel in a standard
way as in \cite[p.~153]{GrMF}\,. We omit the details. Hence we have
$$
II \le C\, Mf(x).
$$
To control $I$ we use \eqref{localstrongcoifman-fefferman}. Hence,
since the support of $f_1$ is contained in $2Q$ we have
$$I\, \le \frac{C}{|Q|} \int_{4Q} |T(f_{1})(y)|\,dy \leq
\frac{C}{|Q|} \int_{4Q} M(f_{1})(y)\,dy \le C\, \frac{C}{|4Q|}
\int_{4Q} M(f)(y)\,dy \le C\, M^2(f)(x).
$$

\begin{proof}[Proof of Theorem \ref{teoremamaximal}]

To prove a) we use part a) of Coifman-Fefferman's Theorem
\ref{CoifFefferman} and part a)  Fefferman-Stein's Theorem
\ref{FeffermanStein}:
$$
\int_{\Rn} \Big(T^*_1\circ T_2(f)(x)\Big)^{p}\, w \le \int_{\Rn}
\Big(M\circ T_2f(x)\Big)^{p}\, w(x)dx
$$
$$
\le C\, \int_{\Rn} \Big(M^{\#}\circ T_2f(x)\Big)^{p}\, w(x)dx \le
C\, \int_{\Rn} (M^{2}f)^{p}\, w
$$
where in the last estimate we have used \eqref{fslogl} in Lemma
\ref{pointwise}\,.  This yields \eqref{coiffeffstrongM2} and
concludes the proof of the first part of the theorem.

To prove \eqref{coiffeffweakM2} we use similar arguments except that
we use part b) of both Theorems \ref{CoifFefferman} and
\ref{FeffermanStein}:
$$\sup_{ t >0}\frac{1}{ \Phi (\frac{1}{t}) }w(\{y\in \Rn\,:\,
|T^*_1\circ T_2f|>t \})\leq \sup_{ t >0}\frac{1}{ \Phi (\frac{1}{t})
} w( \{ y\in\mathbb R ^{n}: M(T_2f)(y) > t \} ).
$$
$$\leq \sup_{ t >0}\frac{1}{ \Phi (\frac{1}{t}) }w(\{y\in \Rn\,:\,
M^{\#}(T_2\textbf{}f)(y)>t \})\leq \sup_{ t >0}\frac{1}{ \Phi
(\frac{1}{t}) } w( \{ y\in\mathbb R ^{n}: M^2(f)(y) > t \} ).
$$

To prove b) in Theorem 2 we use a similar argument. The main
difference is that we use first Cotlar's improved estimate from
Theorem \ref{CotlarPuntual}. Indeed, this is used after an
application of Theorem \ref{CoifFefferman} of Coifman and Fefferman:
$$
\int_{\Rn} \Big(T^*_1\circ T^*_2(f)(x)\Big)^{p}\, w \le \int_{\Rn}
\Big(M\circ T^*_2f(x)\Big)^{p}\, w(x)dx
$$
$$
\leq \int_{\Rn} \Big(M\circ M_{\delta}\circ T_2f(x)\Big)^{p}\,
w(x)dx + \int_{\Rn} M^2f(x)^{p}\, w(x)dx= I+II.
$$
We just need to control $I$. For this we remark that
\begin{equation}\label{pointMdelta}
M\circ M_{\delta}f \leq c_{\delta} Mf(x). \end{equation}
Hence by  Fefferman-Stein's theorem \ref{FeffermanStein}
$$
I\leq C_{\delta}\,\int_{\Rn} \Big(M\circ T_2f(x)\Big)^{p}\, w(x)dx
\leq C_{\delta}\,\int_{\Rn} \Big(M^{\#}\circ T_2f(x)\Big)^{p}\,
w(x)dx
$$
$$
\le C\, \int_{\Rn} \Big(M^{2}f(x)\Big)^{p}\, w(x)dx
$$
where in the last estimate we have used \eqref{fslogl} from Lemma
\ref{pointwise}.

We are left with the proof of \eqref{pointMdelta}. Let $x\in \Rn$
and let $Q=Q(x,r)$ be an arbitrary cube centered at $x$ with
sidelength $r$. We have to show that
\begin{equation*}
\frac{1}{|Q|} \int_{Q} M_{\delta}f(y) \,dy  \le C\, Mf(x).
\end{equation*}
Let $f = f_{1}+f_{2}$, where $f_{1} = f\,  \chi_{2Q}\,.$ \,  We can
estimate the left hand side by a multiple of
$$ \frac{1}{|Q|} \int_{Q} M_{\delta}f_1(y) \,dy
+ \frac{1}{|Q|} \int_{Q} M_{\delta}f_2(y) \,dy = I + II.
$$
To take care of II we use that it is roughly constant on $Q$ by
\cite[p.~299]{GrMF}. Hence we have
$$
II \le C\,M_{\delta}f(x) \leq C\,Mf(x).
$$
To control $I$ we use that $\delta<1$ and that the maximal operator
is bounded on $L^{1/\delta}(\Rn)\,.$  We obtain
$$I\, \le \frac{C_{\delta}}{|Q|} \int_{2Q} |f(y)|\,dy \leq C\, M(f)(x).
$$
This concludes the proof of the first part of b) of Theorem
\ref{teoremamaximal}. The proof of the second part is similar to the
proof of part a)\,, except for the fact that one uses the method we
have just described. We leave the details to the interested reader.

\end{proof}

\begin{proof}[Proof of Corollary \ref{coroLlogLCZO}]

By homogeneity it is enough to assume $t=1$ and hence we just need
to prove
$$w (\{ y\in\Rn: |T^*_1\circ T_2f(y)|> 1 \})\leq C\, \int_{\mathbb
R^n} \Phi (|f(y)|)w(y)dy.$$

Now, $\Phi=t(\log(e+t))\approx t(1+\log^{+}t)$ is submultiplicative,
that is, $\Phi(ab)\le \Phi(a)\,\Phi(b)$, $a,b\geq 0$. In particular,
$\Phi$ is doubling. We have by Theorem \ref{teoremamaximal} and
\eqref{endpointM^2}
\begin{eqnarray*}
w( \{ y\in\mathbb R ^{n}:|T^*_1\circ T_2f(y)|> 1 \})&\leq& C\,
\sup_{t>0}\frac{1}{
\Phi(\frac{1}{t}) }w (\{ y\in\mathbb R ^{n}: |T^*_1\circ T_2f(y)|> t \})\\
&\leq &C\, \sup_{t>0}\frac{1}{
\Phi(\frac{1}{t}) }w (\{ y\in\mathbb R ^{n}:  M^2 f (y)> t \})\\
&\leq &C\, \sup_{t>0}\frac{1}{ \Phi(\frac{1}{t}) }\int_{\mathbb R^n}
\Phi (\frac{|f(y)|}{t})w(y)dy\\ &\leq &C\, \sup_{t>0}\frac{1}{
\Phi(\frac{1}{t}) }\int_{\mathbb R^n} \Phi (|f(y)|)\Phi
(\frac{1}{t})w(y)dy\\ &= &C\, \int_{\mathbb R^n} \Phi
(|f(y)|)w(y)dy\,,
\end{eqnarray*}
which completes the proof.

\end{proof}

\begin{proof}[Proof of inequality \eqref{compostionBeurlingpointwise}]

It is enough by translation invariance to consider $z=0$ in
\eqref{compostionBeurlingpointwise}, that is,
$$B^*(B(f))(0)\leq C\,(B^2)^*(f)(0) + C\,M(f)(0).
$$
Recall that
$$
B^{*}f(z)= \sup_{\epsilon > 0} | B_{\epsilon}f(z)|, \quad z \in \C
\,,
$$
with
$$
B_{\epsilon}f(z)= \int_{| w-z| > \epsilon} f(z-w) \frac{1}{w^2}
\,dw\, .
$$
To prove the inequality at $0$ we use that (see \cite{MV}) for any
$h$
$$B^*(h)= \widetilde{M} (Bh)\,,
$$
where
$$ \widetilde{M} (g)(z)= \sup_{\epsilon>0}\Big|\frac{1}{\pi
\epsilon^2}\int_{D(z,\epsilon)} g(w)\,dA(w)\Big|.
$$
Hence, it is enough to show that
\begin{equation*}\label{claim}
\widetilde{M} (B^2f)(0)\leq C\,(B^2)^*(f)(0) + C\,M(f)(0).
\end{equation*}
By dilation invariance is enough to estimate the integral of
$B^2f$ on the unit disc $D$\,. Clearly
$$
\int_{D} B^2f(w)\,dA(w)= \int f(w)\,B^2(\chi_D)(w)\,dA(w)\,,
$$
and so we need to compute $B^2(\chi_D)$\,. For this we use the basic
property of $B$, namely
$$\frac{\partial}{\partial z} \varphi = -\frac{1}{\pi}
B(\frac{\partial \varphi}{\partial \overline{z}})\,,$$
which holds
for appropriate classes of functions $\varphi$\,.

Integrating the function $\chi_D(z)$ in $\overline{z}$ one gets the
function
$$\varphi(z) = \overline{z}\chi_D(z) + \frac{1}{z} \chi_{D^c}(z) \,,$$
and so
 $$-\frac{1}{\pi} B(\chi_D)(z) = \frac{\partial\varphi}{\partial z} = -
\frac{1}{z^2}\chi_{D^c}(z)\,.$$ Following the same strategy for the
function $-\frac{1}{z^2}\chi_{D^c}(z)$ we get

$$ - \frac{1}{\pi} B^2(\chi_D)(z)=
(-2\frac{\bar{z}}{z^3}+\frac{3}{z^4})\chi_{D^c}$$
and so
$$
- \frac{1}{\pi} \int f(w)\,B^2(\chi_D)(w)\,dA(w) =-2\int_{D^c}
f(w)\,\frac{\bar{w}}{w^3}\,dA(w)+ 3\int_{D^c}
f(w)\,\frac{1}{w^4}\,dA(w).
$$
Last term is bounded by a multiple of $Mf(0)$ since, after putting
the  absolute value inside the integral, one is convolving with a
non-negative decreasing integrable kernel. Alternatively one may
just integrate in dyadic annuli centered at $0\,.$ For the first
term we simply observe that the function $ -2 \frac{\bar{w}}{w^3}$
is the kernel of $B^2$ and hence
$$
 -2 \,\int_{D^c} f(w)\,\frac{\bar{w}}{w^3}\,dA(w)
$$
is the truncation at level $1$ of $B^2f(0)$ (see the definition just
after \eqref{eq2})\,.
\end{proof}

\begin{proof}[Proof of Theorem \ref{higherorderRieszTransforms}]

By dilating and translating it is enough to prove that
$$
|R^1(T(f))(0)|\le C(|S^2f(0)|+ Mf(0)),
$$
where $R^1$ and $S^2$ are the truncations of $R$ and $S$ at levels
$1$ and $2$ respectively (see the definition just after
\eqref{eq2}).

Denote by $K_0, K$ and $K_1$  the kernels of $R$\,,\,\,$ T$ and
$S$ respectively. Let $B$ be the unit ball of $\Rn$\,. It was
shown in \cite{MOV} that, because $R$ is an even  higher order
Riesz transform, its kernel off the unit ball is in the range of
$R$. More precisely,  there exists a polynomial $b$ such that
$$K_0(y)\chi_{B^c}(y)=R(b\chi_{B})(y)\,, \quad y \in \Rn \,.$$
Thus, since $R \circ T = S + c \,I$ \,,
\begin{equation*}
\begin{split}
 R^1(Tf(0)) & = \int_{|y|\ge 1}K_0(y)Tf(y)dy
  \\
 & = \int R(b\chi_{B})(y)Tf(y)dy\\
 & = \int b(y)\chi_{B}(y)Sf(y)dy +c\int b(y)\chi_{B}(y)f(y)dy\\
 & = I+II.
\end{split}
\end{equation*}
Clearly, $II$ is bounded by $C\|b\|_{L^\infty(B)}\,(Mf)(0).$
On the
other hand,
\begin{equation*}
\begin{split}
I & = \int S(b\chi_{B})(y)f(y)dy  \\
& = \int_{2B} S(b\chi_{B})(y)f(y)dy + \int_{(2B)^c}
S(b\chi_{B})(y)f(y)dy \\
& =III + IV.
\end{split}
\end{equation*}
Using Lemma 5 in \cite{MOV} we get that  $III$ is bounded by
$$C Mf(0)
(\|b\|_{L^\infty(B)}+\|b\|_{\operatorname{Lip}(1,B)})\,,
$$
where $\|b \|_{\operatorname{Lip}(1,B)}$ is the Lipschitz semi-norm
of $b$ on $B$\,. Since the kernel $K_1$ of $S$ is smooth off the
origin we have
$$
S(b\chi_{B})(y)= K_1(y)\int b\chi_{B}+
\|b\|_{L^\infty(B)}\,O(\frac{1}{|y|^{n+1}})\,, \quad |y|> 2\,.
$$
Thus,
\begin{equation*}
\begin{split}
|IV| &\le C\|b\|_{\infty} | \int_{(2B)^c} K_1(y)f(y)dy |+ C
\|b\|_{\infty}
\int_{(2B)^c}\frac{|f(y)|}{|y|^{n+1}}dy \\
 & \le C\|b\|_{\infty}(|S^2f(0)|+Mf(0))\,,
 \end{split}
\end{equation*}
which completes the proof.

\end{proof}

\begin{gracies}
 The first, second and fourth authors were partially supported by grants 2009SGR420
(Generalitat de Catalunya) and  MTM2007-60062 (Spanish Ministry of
Science). The third author was partially supported by grants
FQM-1509  (Junta de Andaluc\'{\i}a) and MTM2006-05622 (Spanish Ministry
of Science).

This paper was completed during a research programme in Analysis
held in the spring of 2009 at the ``Centre de Recerca Matem\`{a}tica" in
Barcelona. The authors are grateful to the CRM staff for the fine
organization of this event.
\end{gracies}

\begin{tabular}{l}
Joan Mateu\\
Departament de Matem\`{a}tiques\\
Universitat Aut\`{o}noma de Barcelona\\
08193 Bellaterra, Barcelona, Catalonia\\
{\it E-mail:} {\tt mateu@mat.uab.cat}\\ \\
Joan Orobitg\\
Departament de Matem\`{a}tiques\\
Universitat Aut\`{o}noma de Barcelona\\
08193 Bellaterra, Barcelona, Catalonia\\
{\it E-mail:} {\tt orobitg@mat.uab.cat}\\ \\
Carlos P\'{e}rez\\
Departamento de An\'{a}lisis Matem\'{a}tico \\
 Universidad de Sevilla\\
 41012 Sevilla, Spain\\
{\it E-mail:} {\tt carlosperez@us.es}\\ \\
Joan Verdera\\
Departament de Matem\`{a}tiques\\
Universitat Aut\`{o}noma de Barcelona\\
08193 Bellaterra, Barcelona, Catalonia\\
{\it E-mail:} {\tt jvm@mat.uab.cat}
\end{tabular}

\end{document}